\documentclass[10pt]{article}
\usepackage{latexsym, amssymb, amsmath, amsthm, amscd}
\usepackage[all,cmtip]{xy}
\usepackage{amsmath}
\usepackage{amsthm}
\usepackage{amssymb}
\usepackage{amsfonts}
\usepackage{amsrefs}
\usepackage{color,authblk}
\usepackage{tikz}

\usepackage{amscd} 
\usepackage{comment}
\usepackage{mathrsfs}
\usepackage[all]{xy}
\usepackage{graphicx}
\usepackage{authblk}
\usepackage{hyperref}
\usepackage{fullpage}
\usepackage[all,cmtip]{xy}
\usepackage{relsize}
\usepackage{amsmath,amscd}
\usepackage{amssymb}
\usepackage{stackrel}
\usepackage{tikz-cd}
\usepackage{tikz}
\usetikzlibrary{arrows}
\usetikzlibrary{matrix}
\usepackage[mathscr]{euscript}

\usepackage[all]{xy}

\usepackage{hyperref}

\numberwithin{equation}{section} \DeclareMathSizes{2}{10}{12}{13}
\makeatletter
\newcommand*{\doublerightarrow}[2]{\mathrel{
  \settowidth{\@tempdima}{$\scriptstyle#1$}
  \settowidth{\@tempdimb}{$\scriptstyle#2$}
  \ifdim\@tempdimb>\@tempdima \@tempdima=\@tempdimb\fi
  \mathop{\vcenter{
    \offinterlineskip\ialign{\hbox to\dimexpr\@tempdima+1em{##}\cr
    \rightarrowfill\cr\noalign{\kern.5ex}
    \rightarrowfill\cr}}}\limits^{\!#1}_{\!#2}}}

\newcommand{\leftrarrows}{\mathrel{\raise.75ex\hbox{\oalign{%
  $\scriptstyle\leftarrow$\cr
  \vrule width0pt height.5ex$\hfil\scriptstyle\relbar$\cr}}}}
\newcommand{\lrightarrows}{\mathrel{\raise.75ex\hbox{\oalign{%
  $\scriptstyle\relbar$\hfil\cr
  $\scriptstyle\vrule width0pt height.5ex\smash\rightarrow$\cr}}}}
\newcommand{\Rrelbar}{\mathrel{\raise.75ex\hbox{\oalign{%
  $\scriptstyle\relbar$\cr
  \vrule width0pt height.5ex$\scriptstyle\relbar$}}}}

\makeatletter
\def\leftrightarrowsfill@{\arrowfill@\leftrarrows\Rrelbar\lrightarrows}
\newcommand{\xleftrightarrows}[2][]{\ext@arrow 3399\leftrightarrowsfill@{#1}{#2}}
\makeatother

\newtheorem{thm}{Proposition}[section]

\newtheorem{rem}[thm]{Remark}

\newtheorem{lem}[thm]{Lemma}
\newtheorem{defn}[thm]{Definition}

\title{Chow groups of ind-schemes and extensions of Saito's filtration}
\author{Abhishek Banerjee}
\date{}

\begin{document}

\maketitle

\medskip
\centerline{\emph{Department of Mathematics, Indian Institute of Science, Bangalore, India.}}
\centerline{\emph{Email: abhishekbanerjee1313@gmail.com}}

\medskip

\medskip

\begin{abstract} Let $K$ be a field of characteristic zero and let $Sm/K$ be the category
of smooth and separated schemes over $K$. For an ind-scheme $\mathcal X$ (and more
generally for any presheaf of sets on $Sm/K$), we define its Chow groups
$\{CH^p(\mathcal X)\}_{p\in \mathbb Z}$. We also introduce Chow groups $\{\mathcal{CH}^p(\mathcal G)\}_{p\in \mathbb Z}$
for a presheaf with transfers $\mathcal G$ on  $Sm/K$.  Then, we show that we have natural isomorphisms of Chow groups
\begin{equation*}
CH^p(\mathcal X)\cong \mathcal{CH}^p(Cor(\mathcal X))\qquad\forall\textrm{ }p
\in \mathbb  Z\end{equation*} where $Cor(\mathcal X)$ is the presheaf with transfers
that associates to any $Y\in Sm/K$ the collection of finite correspondences from $Y$ to $\mathcal X$. Additionally, when $K=\mathbb C$, we show that Saito's filtration on
the Chow groups of a smooth projective scheme can be extended to the  Chow groups
$CH^p(\mathcal X)$ and more generally, to the Chow groups of an arbitrary presheaf
of sets on $Sm/\mathbb C$.  Similarly, there exists an extension of Saito's filtration to the Chow groups
of a presheaf with transfers on $Sm/\mathbb C$. Finally, when the ind-scheme $\mathcal X$ is ind-proper,
we show that the isomorphism $CH^p(\mathcal X)\cong \mathcal{CH}^p(Cor(\mathcal X))$ is actually a filtered isomorphism. 
\end{abstract}

\medskip
{\bf MSC(2010) Subject Classification: 14C15, 14C25}


\medskip

\medskip

\section{Introduction}

\medskip

\medskip

In this paper, we define Chow groups of ind-schemes, using the bivariant Chow theory of Fulton and MacPherson \cite{FM}. 
Various aspects of the theory of  ind-schemes have been studied in detail by several authors. For instance, the derived algebraic geometry for ind-schemes
has been developed by Gaitsgory and Rozenblyum (see \cite{GR1}, \cite{GR2}, \cite{GR3}). Further, ind-coherent sheaves over ind-schemes have been studied
by Gaitsgory \cite{Gait}. Sheaves over ind-schemes and ind-stacks have been used by Bezrukavnikov, Kazhdan and Varshavsky in \cite{BKY} to study problems in
representation theory. Ind-schemes have also been applied to the study of $D$-modules in papers by Kapranov and Vasserot \cite{KV1}, \cite{KV2}, \cite{KV3}, \cite{KV4}. As such, we hope that this paper is the first step towards developing the intersection theory for ind-schemes. 

\medskip
Let $K$ be a field of characteristic zero. Then, Fulton and MacPherson \cite{FM} have defined the bivariant Chow groups
$CH^p(f:X'\longrightarrow X)$, $p\in \mathbb Z$ for a morphism $f:X'\longrightarrow X$ of schemes over $K$.  Let $g:X''\longrightarrow X$ be a morphism and consider the fiber square:
\begin{equation}
\begin{CD}
X''' @>f'>> X'' \\
@Vg'VV @VgVV \\
X' @>f>> X \\
\end{CD}
\end{equation} Then, a bivariant class $c\in CH^p(f:X'\longrightarrow X)$ is that which associates to each such  morphism $g:X''\longrightarrow X$ of schemes, a family of homomorphisms:
\begin{equation}
c^k_g:CH_k(X'')\longrightarrow CH_{k-p}(X''') \qquad \forall\textrm{ }k\in \mathbb Z
\end{equation}
that is well behaved with respect to flat pullbacks, proper pushforwards and refined Gysin morphisms (see \cite{FM}).   In particular, for a scheme $X$  over $K$, the operational Chow cohomology groups
of $X$ are defined as:
\begin{equation}
CH^p(X):=CH^p(1:X\longrightarrow X) \qquad \forall\textrm{ }p\in \mathbb Z
\end{equation} Then, for any morphism $h:Y\longrightarrow X$, the  formalism of the bivariant theory of Fulton and MacPherson allows us to define
a pullback $h^*:CH^p(X)\longrightarrow CH^p(Y)$ on the operational Chow groups. We now restrict ourselves to the category $Sm/K$ of smooth and separated schemes over $K$. Then, if
$Y$, $X\in Sm/K$, we know that a correspondence $Z$ 
from $Y$ to $X$ induces a pullback on Chow cohomology groups. The starting point of this paper is the  aim to
construct a bivariant Chow theory that naturally incorporates these additional pullback maps induced by correspondences. 
More generally, we will construct a bivariant formalism to define Chow groups of presheaves with transfers. Of particular interest
are the Chow groups of presheaves with transfers $C_X$ corresponding to some $X\in Sm/K$ (for any $Y\in Sm/K$, $C_X(Y)$ is the abelian
group of all finite correspondences from $Y$ to $X$). More generally, we will study Chow groups of the presheaves
with transfers $Cor(\mathcal X)$  corresponding to some ind-scheme $\mathcal X$  over $K$.  
 By an ind-scheme over $K$, we will mean an ind-object of the category $Sm/K$; in other words, the category
$Ind-(Sm/K)$ of ind-schemes over $K$ is the ``free cocompletion of $Sm/K$ with respect to filtered colimits''. It is clear
that the category $Sm/K$ embeds into the category $Ind-(Sm/K)$ of ind-schemes.

\medskip Let $\mathcal X$ be an ind-scheme over $K$. We proceed to define the Chow groups
$\{CH^p(\mathcal X)\}_{p\in \mathbb Z}$ of $\mathcal X$ in two different ways. First, we can define $\{CH^p(\mathcal X)\}_{p\in \mathbb Z}$ to be the Chow groups of the presheaf
of sets on $Sm/K$ defined by the ind-scheme $\mathcal X$. These Chow groups are defined by adapting our definition
in \cite{AB6} and in the spirit of the bivariant Chow theory of Fulton and MacPherson \cite{FM}. On the other hand, if we enlarge
the category $Sm/K$ by adding correspondences, we can consider the presheaf with transfers $Cor(\mathcal X)$ associated
to $\mathcal X$. For any $Y\in Sm/K$, $Cor(\mathcal X)(Y)$ may be seen as the collection of finite correspondences from 
$Y$ to the ind-scheme $\mathcal X$. Then, we can define Chow groups of $\mathcal X$ by considering the Chow groups
$\{\mathcal{CH}^p(Cor(\mathcal X))\}_{p\in \mathbb Z}$ 
of the presheaf with transfers $Cor(\mathcal X)$. Our first main result is to show that the Chow groups of $\mathcal X$ obtained from these two
approaches are isomorphic to each other, i.e., $CH^p(\mathcal X)\cong \mathcal{CH}^p(Cor(\mathcal X))$ $\forall$ $p\in \mathbb Z$. When $K=\mathbb C$,  we  extend Saito's filtration \cite{Saito}
on the Chow groups of a smooth projective scheme over $\mathbb C$ to the Chow groups $CH^p(\mathcal X)$ of the ind-scheme $\mathcal X$ considered as a presheaf
of sets on $Sm/\mathbb C$. Similarly,  there exists an extension of Saito's filtration to the Chow groups $\mathcal{CH}^p(Cor(\mathcal X))$ of the presheaf with transfers
$Cor(\mathcal X)$. When the ind-scheme $\mathcal X$ is also ind-proper, our second main result is that the isomorphism $CH^p(\mathcal X)\cong \mathcal{CH}^p(Cor(\mathcal X))$  mentioned above can be promoted
to a filtered isomorphism with respect to these filtrations.

\medskip
We now describe the structure of the paper in more detail: we start in Section 2 by considering the 
category $Cor_K$ with the same objects as $Sm/K$ and in which the collection of morphisms from $Y$ to 
$X$ is the  abelian group $Cor(Y,X)$ of finite correspondences from $Y$ to $X$ (see Definition \ref{D1}).   Then the category $Presh(Cor_K)$ of  presheaves with transfers is the category of additive contravariant functors from $Cor_K$ to the category $\mathbf{Ab}$  of abelian
groups. For any $X\in Cor_K$, the association $Y\mapsto Cor(Y,X)$ $\forall$ $Y\in Cor_K$ determines a presheaf with transfers which we denote by $C_X$. Let $Presh(Sm/K)$ denote the category of presheaves of sets on 
$Sm/K$. We start by considering
the  functor $Cor$ that associates to any presheaf $\mathcal F
\in Presh(Sm/K)$ a presheaf with transfers $Cor(\mathcal F)$. In particular, if $\mathcal F=h_X$, the presheaf of sets
on $Sm/K$ represented by some $X\in Sm/K$, we have $Cor(h_X)=C_X$. We mention that this functor of adding transfers has already been studied by
Cisinski and D\'{e}glise \cite[Section 10.3-10.4]{CD}. The  functor $Cor$ is left adjoint to the forgetful functor, i.e.,
\begin{equation}
Mor_{Presh(Cor_K)}(Cor(\mathcal F),\mathcal G)\cong Mor_{Presh(Sm/K)}
(\mathcal F,\mathcal G)
\end{equation} for any $\mathcal F\in Presh(Sm/K)$ and any $\mathcal G\in Presh(Cor_K)$. Applying the functor
$Cor$ to an ind-scheme $\mathcal X$, we obtain a presheaf with transfers $Cor(\mathcal X)$. Then, for any
scheme $Y\in Cor_K$,  $Cor(\mathcal X)(Y)$ may be seen as the collection of finite correspondences from
$Y$ to the ind-scheme $\mathcal X$. For any scheme $X\in Sm/K$, we know that the presheaf with transfers
$C_X$ is a sheaf for the \'{e}tale topology on $Sm/K$. Then, in Section 2 we show that for an ind-scheme 
$\mathcal X$, the presheaf with transfers $Cor(\mathcal X)$ is also an \'{e}tale sheaf on $Sm/K$.

\medskip
 In Section 3, we take two approaches to defining the Chow groups of an ind-scheme $\mathcal X$. First,  we define the Chow groups
$\{CH^p(\mathcal F)\}_{p\in \mathbb Z}$ for any presheaf $\mathcal F\in Presh(Sm/K)$ 
as well as the Chow groups $\{\mathcal{CH}^p(\mathcal G)\}_{p\in \mathbb Z}$
for any presheaf with transfers $\mathcal G\in Presh(Cor_K)$   by modifying the definition
 in \cite{AB6} (see Definition \ref{D31} and Definition \ref{D32} respectively).  In 
particular, for the presheaf $h_X$ on $Sm/K$ represented by some $X\in Sm/K$, we show that the Chow
groups $\{CH^p(h_X)\}_{p\in \mathbb Z}$ are naturally isomorphic to the usual Chow
groups of $X$. Given $X\in Sm/K$, we can also consider the Chow groups of the presheaf with transfers $C_X$ defined by
setting $C_X(Y):=Cor(Y,X)$ for any $Y\in Cor_K$. Then, the first main result of Section 3 is that for any $X\in Sm/K$, we 
have isomorphisms: 
\begin{equation}\label{ZYZ1.2}
CH^p(X)\cong CH^p(h_X)\cong \mathcal{CH}^p(C_X)\qquad \forall\textrm{ }p\in \mathbb Z
\end{equation} Thereafter, we consider the case of an ind-scheme $\mathcal X$ over $K$. If we treat $\mathcal X$
as a presheaf of sets on $Sm/K$, we have the corresponding Chow groups $\{CH^p(\mathcal X)\}_{p\in \mathbb Z}$
of $\mathcal X$. We also consider the Chow groups of the presheaf with transfers $Cor(\mathcal X)$ associated to
$\mathcal X$. Then, the main result of Section 3 is that for any $\mathcal X\in Ind-(Sm/K)$, we have isomorphisms:
\begin{equation}\label{ZYZ1.3}
CH^p(\mathcal X)\cong \mathcal{CH}^p(Cor(\mathcal X))\qquad\forall\textrm{ }p
\in \mathbb Z
\end{equation}

\medskip
In the final section, i.e., in Section 4, we assume that $K=\mathbb C$. For any smooth
projective scheme $Y$ over $\mathbb C$, Saito \cite{Saito} has defined a descending
filtration on its Chow groups, which is a candidate for the conjectural
Bloch-Beilinson filtration (see, for instance, Levine \cite[$\S$ 7]{Levine2}). Then, we extend Saito's filtration 
on the Chow groups of a smooth projective scheme over $\mathbb C$
to the Chow groups of any presheaf $\mathcal F$ on $Sm/\mathbb C$ (see Definition \ref{D62}). 
We mention here that if we work with presheaves on the category $SmProj/\mathbb C$ of smooth and projective schemes,  instead of $Sm/\mathbb C$, we have introduced a similar
notion for an extended Saito filtration on Chow groups in \cite[D\'{e}finition 3.2]{AB7}. In fact, when
working with presheaves on $SmProj/\mathbb C$, we can also obtain an extended Saito filtration on
higher Chow groups (see \cite{AB7}). For our earlier work on higher bivariant Chow groups, we refer
the reader to \cite{AB0}. 

\medskip In particular, we know from 
Section 3 that for any $X\in Sm/\mathbb C$, the Chow groups $\{CH^p(h_X)\}_{p\in \mathbb Z}$
of the presheaf $h_X$ on $Sm/\mathbb C$ represented by $X$ are isomorphic to the
usual Chow groups of $X$. It follows that when $X$ is any smooth and separated (but not necessarily projective) scheme
over $\mathbb C$, we have an extension of Saito's filtration to the Chow groups of $X$:
\begin{equation}\label{ZYZ4.18}
CH^p(X)\cong CH^p(h_X)=F^0CH^p(h_X)\supseteq F^1CH^p(h_X)\supseteq F^2CH^p(h_X)
\supseteq \dots 
\end{equation}
 When $X$ is in fact smooth and projective, we show that our filtration 
$\{F^rCH^p(h_X)\}_{r\geq 0}$ on $CH^p(h_X)\cong CH^p(X)$  actually coincides with Saito's filtration on the Chow
group $CH^p(X)$. Further, we show that we can extend Saito's filtration to the Chow groups of presheaves with
transfers (see Definition \ref{D66}) and in particular to the Chow groups $\mathcal{CH}^p(C_X)$ for any $X\in Sm/\mathbb C$:
\begin{equation}\label{ZYZ4.19}
\mathcal{CH}^p(C_X)=F^0\mathcal{CH}^p(C_X)\supseteq F^1\mathcal{CH}^p(C_X)\supseteq F^2\mathcal{CH}^p(C_X)
\supseteq \dots 
\end{equation} As mentioned in \eqref{ZYZ1.2} above, it is shown in Section 3 that $CH^p(h_X)\cong \mathcal{CH}^p(C_X)$, 
$\forall$ $p\in \mathbb Z$. Then, the first main result of Section 4 is that when $X\in Sm/\mathbb C$ is also proper,
the filtrations in \eqref{ZYZ4.18} and \eqref{ZYZ4.19} on $CH^p(h_X)$ and $\mathcal{CH}^p(C_X)$ respectively are isomorphic. Again, for an ind-scheme $\mathcal X$ over $\mathbb C$, we obtain an extension of Saito's filtration to the Chow groups
of $\mathcal X$:
\begin{equation}\label{ZYZ1.6}
CH^p(\mathcal X)=F^0CH^p(\mathcal X)\supseteq F^1CH^p(\mathcal X)\supseteq F^2CH^p(\mathcal X)
\supseteq \dots 
\end{equation} If we consider the presheaf with transfers $Cor(\mathcal X)$ associated to the ind-scheme $\mathcal X$, we obtain a filtration
on the Chow groups of $Cor(\mathcal X)$:
\begin{equation}\label{ZYZ1.7}
\mathcal{CH}^p(Cor(\mathcal X))=F^0\mathcal{CH}^p(Cor(\mathcal X))\supseteq F^1\mathcal{CH}^p(Cor(\mathcal X))\supseteq F^2\mathcal{CH}^p(Cor(\mathcal X))
\supseteq \dots 
\end{equation} As mentioned in \eqref{ZYZ1.3} above, it is shown in Section 3 that $CH^p(\mathcal X)\cong \mathcal{CH}^p(Cor(\mathcal X))$, $\forall$ $p\in \mathbb Z$. Then, the final objective of Section 4 is to show that when the ind-scheme $\mathcal X$ is
also ind-proper, we have isomorphisms:
\begin{equation}
F^rCH^p(\mathcal X)\cong F^r\mathcal{CH}^p(Cor(\mathcal X)) \qquad \forall\textrm{ }r\geq 0
\end{equation} In other words, for any ind-scheme $\mathcal X$ that is also ind-proper, the isomorphism in \eqref{ZYZ1.3} is a filtered isomorphism with respect to the filtrations in \eqref{ZYZ1.6} and \eqref{ZYZ1.7}.

\medskip

\medskip

\section{The functor $\mathbf{Cor}$ on   presheaves and ind-schemes}

\medskip

\medskip
Let $K$ be a field of characteristic zero and let $Sm/K$ denote the category of 
smooth and separated schemes over
$K$. For any scheme $X\in Sm/K$, we let $h_X$ denote the presheaf  on $Sm/K$
represented by $X$.

\medskip
\begin{defn}\label{D1} (see \cite[Definition 1.1]{WMV}) Let $X$, $Y\in Sm/K$ and let $Y=\underset{i\in I}{\coprod} Y_i$ be the decomposition of $Y$ into its connected components. For any given $i\in I$, an elementary correspondence from 
$Y_i$ to $X$ is an irreducible closed subset $Z\subseteq Y_i\times X$ whose associated
integral subscheme is finite and surjective over $Y_i$. Let $Cor(Y_i,X)$ be the free abelian group generated by all elementary  correspondences from $Y_i$ to $X$. Then, the collection
$Cor(Y,X)$  of all finite correspondences from $Y$ to $X$  is the direct sum  $Cor(Y,X):=\underset{i\in I}{\oplus}Cor(Y_i,X)$ of abelian groups.   
\end{defn}

\medskip
\noindent  For any morphism $f:Y\longrightarrow X$ in $Sm/K$ with $Y$ connected, the graph $\Gamma_f$ 
of the morphism $f$ determines a finite correspondence from $Y$ to $X$. Further, it is well known that there exists a composition of correspondences 
\begin{equation}\label{2.1}
Cor(Z,Y)\otimes Cor(Y,X)\longrightarrow Cor(Z,X) \qquad (u,v)\mapsto v\circ u
\end{equation} for any $X$, $Y$, $Z\in Sm/K$. For more on the composition
of correspondences, we refer the reader to the work of Suslin and Voevodsky\cite{SV}. It follows 
from \eqref{2.1} that for any $X\in Sm/K$, we have
a presheaf $C_X$ on $Sm/K$ defined by $C_X(Y):=Cor(Y,X)$ for any $Y\in Sm/K$. We notice
that $C_X$ is actually a presheaf of abelian groups on $Sm/K$. Accordingly, 
for any $Y$, $Z\in Sm/K$ and given a correspondence $u\in Cor(Z,Y)$, it follows from 
\eqref{2.1} that there exists an induced morphism 
\begin{equation}\label{2.2}
\begin{CD}
C_X(u):C_X(Y)=Cor(Y,X)@>\_\_\circ u>> Cor(Z,X)=C_X(Z)\\
\end{CD}
\end{equation} of abelian groups. 

\medskip
\begin{defn}\label{D2} (see \cite[Definition 2.1]{WMV}) Let $Cor_K$ denote the additive category
consisting of the same objects as $Sm/K$ and whose morphisms from $Y$ to $X$ are the
elements of $Cor(Y,X)$. A presheaf with transfers is an additive contravariant functor
$\mathcal F:Cor_K\longrightarrow \mathbf{Ab}$, where $\mathbf{Ab}$ denotes the category
of abelian groups. The category of presheaves with transfers will be denoted by
$Presh(Cor_K)$.  
\end{defn} 

\medskip
From \eqref{2.2}, it follows that for any $X\in Sm/K$, we have an associated presheaf with transfers $C_X\in Presh(Cor_K)$. We let $Presh(Sm/K)$ denote the category of presheaves of sets
on $Sm/K$. We will now define a functor $Cor$ that associates to any object $\mathcal F
\in Presh(Sm/K)$ a presheaf with transfers $Cor(\mathcal F)$. For any $\mathcal F
\in Presh(Sm/K)$,  we consider the inductive system $Dir(\mathcal F)$ whose
objects are pairs $(X,x)$ with $X\in Sm/K$ and $x\in \mathcal F(X)$. A morphism
$f:(Y,y)\longrightarrow (X,x)$ in $Dir(\mathcal F)$ is given by a morphism
$f:Y\longrightarrow X$ in $Sm/K$ such that $\mathcal F(f)(x)=y$, where $\mathcal F(f)$
denotes the induced map $\mathcal F(f):\mathcal F(X)\longrightarrow \mathcal F(Y)$. From
Yoneda lemma, it now follows that the presheaf $\mathcal F$ can be expressed
as the colimit:
\begin{equation}\label{2.3}
\mathcal F=\underset{(X,x)\in Dir(\mathcal F)}{colim}\textrm{ }h_X
\end{equation} We now consider the functor
\begin{equation}\label{2.4}
Cor:Presh(Sm/K)\longrightarrow Presh(Cor_K)\qquad Cor(\mathcal F):=\underset{(X,x)\in Dir(\mathcal F)}{colim}\textrm{ }C_X
\end{equation} Since each of the objects $C_X\in Presh(Cor_K)$, it follows from
\eqref{2.4} that the colimit $Cor(\mathcal F)\in Presh(Cor_K)$. We mention here that this functor of adding transfers has 
already been studied by Cisinski and D\'{e}glise \cite[Section 10.3-10.4]{CD}. We will now study below the properties of this functor $Cor:Presh(Sm/K)\longrightarrow Presh(Cor_K)$ that are required for the purposes of this paper.

\medskip
\begin{thm}\label{P1} Let $X$ be an object of $Sm/K$ and let $h_X$ denote the presheaf on $Sm/K$ represented
by $X$. Then, $Cor(h_X)=C_X$. 
\end{thm}

\begin{proof} For the presheaf $h_X$ on $Sm/K$, it is clear that the associated direct
system $Dir(h_X)$ has a final object $(X,1)$, where $1\in h_X(X)$ denotes the identity morphism
$1:X\longrightarrow X$. It follows that :
\begin{equation}
Cor(h_X)=\underset{(Y,y)\in Dir(h_X)}{colim}\textrm{ }C_Y = C_X
\end{equation}

\end{proof}

\medskip
\begin{thm}\label{P2} The functor $Cor:Presh(Sm/K)\longrightarrow Presh(Cor_K)$ is left adjoint
to the forgetful functor, i.e., for any $\mathcal F\in Presh(Sm/K)$, $\mathcal 
G\in Presh(Cor_K)$, we have natural isomorphisms:
\begin{equation}\label{gopinath}
Mor_{Presh(Cor_K)}(Cor(\mathcal F),\mathcal G)\cong Mor_{Presh(Sm/K)}(\mathcal F,
\mathcal G)
\end{equation}
\end{thm}

\begin{proof} First, we consider the case with $\mathcal F=h_X$ for some 
$X\in Sm/K$. From Proposition \ref{P1}, we know that $Cor(h_X)=C_X$. In that case,
we need to show that
\begin{equation}\label{ttr}
Mor_{Presh(Cor_K)}(C_X,\mathcal G)\cong Mor_{Presh(Sm/K)}(h_X,
\mathcal G)\cong \mathcal G(X)
\end{equation} Now, for
any morphism $f:C_X\longrightarrow \mathcal G$ in 
$Presh(Cor_K)$, we have an associated morphism
\begin{equation}\label{e2.8bq}
h_X\overset{\Delta_X}{\longrightarrow} C_X\overset{f}{\longrightarrow} \mathcal G
\end{equation} in $Presh(Sm/K)$. Here $\Delta_X:h_X\longrightarrow C_X$ is the morphism
of presheaves on $Sm/K$ corresponding (by Yoneda lemma) to the element of $Cor(X,X)=C_X(X)$
defined by the identity morphism $1:X\longrightarrow X$. Conversely,
choose any element $x\in \mathcal G(X)$ corresponding to a morphism $x:h_X\longrightarrow 
\mathcal G$ of presheaves on $Sm/K$. Then, since $\mathcal G$ is actually a presheaf with transfers, for any $Y\in Sm/K$, each correspondence $c\in Cor(Y,X)=C_X(Y)$ induces
a morphism
\begin{equation}
\mathcal G(c):\mathcal G(X)\longrightarrow \mathcal G(Y)
\end{equation} It follows that for any element $x\in \mathcal G(X)=Mor_{Presh(Sm/K)}(h_X,\mathcal G)$ we have an associated morphism $C_X\longrightarrow \mathcal G$ of presheaves with transfers determined by the family of morphisms  
\begin{equation}\label{e2.10bq}
C_X(Y)\longrightarrow \mathcal G(Y) \qquad c\mapsto \mathcal G(c)(x) \qquad
\forall\textrm{ }Y\in Cor_K
\end{equation} It is clear that the associations  in \eqref{e2.8bq} and \eqref{e2.10bq} are inverse to each other, thus
proving the isomorphism $Mor_{Presh(Cor_K)}(C_X,\mathcal G)\cong Mor_{Presh(Sm/K)}(h_X,
\mathcal G)\cong \mathcal G(X)$ in \eqref{ttr}.

\medskip
More generally, for any presheaf $\mathcal F\in Presh(Sm/K)$, we consider the system
$Dir(\mathcal F)$ as defined above. It follows that:
\begin{equation}
\begin{array}{ll}
Mor_{Presh(Cor_K)}(Cor(\mathcal F),\mathcal G) & \cong Mor_{Presh(Cor_K)}\left(\underset{(X,x)\in Dir(\mathcal F)}{colim}\textrm{ }C_X,\mathcal G\right) \\
& \cong \underset{(X,x)\in Dir(\mathcal F)}{lim}\textrm{ }Mor_{Presh(Cor_K)}(C_X,\mathcal G) \\
& \cong \underset{(X,x)\in Dir(\mathcal F)}{lim}\textrm{ }Mor_{Presh(Sm/K)}(h_X,\mathcal G) \\
&\cong Mor_{Presh(Sm/K)}\left(\underset{(X,x)\in Dir(\mathcal F)}{colim}\textrm{ }h_X,\mathcal G\right) \\
& \cong Mor_{Presh(Sm/K)}(\mathcal F,\mathcal G)\\
\end{array}
\end{equation}
This proves the result.

\end{proof}

\medskip
For any presheaf $\mathcal F$
of sets on $Sm/K$, let $\mathcal F_{et}$ denote
the sheafification of $\mathcal F$ for the  \'{e}tale topology on $Sm/K$. If 
$\mathcal F$ also carries transfer maps, i.e., $\mathcal F$ is actually a presheaf with transfers, then we know that  the transfer maps lift to the \'{e}tale sheafification 
$\mathcal F_{et}$ of $\mathcal F$ (see \cite[$\S$ 6]{WMV}). In fact, the same holds for the Nisnevich site, i.e., 
if $\mathcal F$ is a presheaf with transfers, the transfer maps lift to the Nisnevich sheafification $\mathcal F_{Nis}$ of
$\mathcal F$ (see \cite[$\S$ 13]{WMV}). However, this is not always true for the
sheafification of $\mathcal F$ with respect to the Zariski site.  In this paper, unless otherwise mentioned, by a sheaf
on $Sm/K$, we will always mean a sheaf for the \'{e}tale topology.

\medskip
 We now let $Sh_{et}(Sm/K)$
denote the category of sheaves (of sets) for the \'{e}tale topology on $Sm/K$. Further, we
let $Sh_{et}(Cor_K)$ denote the full subcategory of $Presh(Cor_K)$ consisting of
presheaves with transfers that are sheaves (of abelian groups) for the \'{e}tale site of $Sm/K$. The objects of the category $Sh_{et}(Cor_K)$ are referred to as sheaves with transfers. We now define the functor
\begin{equation}
Cor_{et}:Presh(Sm/K)\longrightarrow Sh_{et}(Cor_K)\qquad Cor_{et}(\mathcal F):=
(Cor(\mathcal F))_{et}
\end{equation} From the discussion above, we know that the \'{e}tale sheafification
$Cor_{et}(\mathcal F)$ 
of $Cor(\mathcal F)$ actually defines a sheaf with transfers. In particular, $Cor_{et}:
Presh(Sm/K)\longrightarrow Sh_{et}(Cor_K)$ restricts to a functor on the full subcategory 
$Sh_{et}(Sm/K)$ of $Presh(Sm/K)$ which we continue to denote by $Cor_{et}$. We now have the following result.

\medskip
\begin{thm}  The functor $Cor_{et}:Sh_{et}(Sm/K)\longrightarrow Sh_{et}(Cor_K)$ is left adjoint
to the forgetful functor, i.e., for any $\mathcal F\in Sh_{et}(Sm/K)$, $\mathcal 
G\in Sh_{et}(Cor_K)$, we have natural isomorphisms:
\begin{equation}
Mor_{Sh_{et}(Cor_K)}(Cor_{et}(\mathcal F),\mathcal G)\cong Mor_{Sh_{et}(Sm/K)}(\mathcal F,
\mathcal G)
\end{equation}

\end{thm}

\begin{proof} Since $Sh_{et}(Sm/K)$ is a full subcategory of $Presh(Sm/K)$, we have \begin{equation}\label{v45}
Mor_{Sh_{et}(Sm/K)}(\mathcal F,\mathcal G)= Mor_{Presh(Sm/K)}(\mathcal F,\mathcal G)
\end{equation} Now, using the adjointness of functors in Proposition \ref{P2}, we obtain natural 
isomorphisms
\begin{equation}\label{v46}
Mor_{Presh(Cor_K)}(Cor(\mathcal F),\mathcal G)\cong Mor_{Presh(Sm/K)}(
\mathcal F,\mathcal G)
\end{equation} Finally, using \cite[Corollary 6.18]{WMV}, we have natural isomorphisms:
\begin{equation}\label{v47}
Mor_{Sh_{et}(Cor_K)}(Cor_{et}(\mathcal F),\mathcal G)\cong Mor_{Presh(Cor_K)}(Cor(\mathcal F),\mathcal G)
\end{equation} Combining \eqref{v45}, \eqref{v46} and \eqref{v47}, we have the result. 

\end{proof}

\medskip
We will now apply the $Cor$ functor to ind-schemes. By an ind-scheme over $K$, we will mean an ind-object
of the category $Sm/K$ (see Definition \ref{D31pz} below). For a given scheme $X\in Sm/K$, the contravariant functor $C_X:
Cor_K
\longrightarrow \mathbf{Ab}$ associates
to any $Y\in Sm/K$ the collection of finite correspondences from $Y$  to $X$. Then, it is clear that for an ind-scheme
$\mathcal X$ and any object $Y\in Sm/K$, the abelian group $Cor(\mathcal X)(Y)$ may be seen as the collection of ``finite correspondences
from $Y$ to the ind-scheme $\mathcal X$''. Further, we know that for a scheme $X\in Sm/K$, the presheaf with
transfers $C_X$ is actually a sheaf for the \'{e}tale topology on $Sm/K$ (see \cite[Lemma 6.2]{WMV}). We will now show that the same holds 
for ind-schemes, i.e., for an ind-scheme $\mathcal X$, the presheaf with transfers $Cor(\mathcal X)$ is actually
an \'{e}tale sheaf on $Sm/K$.

\medskip
\begin{defn}\label{D31pz}  Let $\mathcal X$ be a presheaf of sets on $Sm/K$. We will say that $\mathcal X$
is an ind-scheme if $\mathcal X$ is a flat functor from $(Sm/K)^{op}$ to the category of 
sets, i.e., the inductive system
\begin{equation}\label{zqap}
Dir(\mathcal X)=\{(X,x)| \textrm{ }x\in \mathcal X(X),X\in Sm/K\}
\end{equation} is filtered. Equivalently, $\mathcal X$ is said to be an ind-scheme if the presheaf $\mathcal X$  can be represented
as a filtered colimt of representable presheaves on $Sm/K$. We will denote the category of ind-schemes by $Ind-(Sm/K)$. 

\end{defn}

\medskip
\begin{lem}\label{0L5} Let $\{\mathcal F_i\}_{i\in I}$ be a filtered inductive system
of objects in $Sh_{et}(Cor_K)$. Let $\mathcal F$ denote the presheaf colimit of
the system, i.e., 
\begin{equation}\label{213br}
\mathcal F(X):=colim_{i\in I}\mathcal F_i(X) \qquad \forall \textrm{ }X\in Cor_K
\end{equation} Then, $\mathcal F$ is a sheaf for the \'{e}tale topology on
$Sm/K$, in other words $\mathcal F\in Sh_{et}(Cor_K)$. 
\end{lem}

\begin{proof} Let $\{X_j\longrightarrow X\}_{j\in J}$ be an \'{e}tale covering
of $X$ in $Sm/K$. We need to show that
\begin{equation}\label{214br}
\mathcal F(X)=lim(\prod_{j\in J}\mathcal F(X_j)\overset{\longrightarrow}{\underset{
\longrightarrow}{ }}\prod_{j,j'\in J}\mathcal F(X_j\times_XX_{j'}))
\end{equation} Since $X$ is smooth, it follows that $X$ is of finite type over $K$ 
(see \cite[$\S$ 17.3.1]{EGA4}) and hence quasi-compact. Therefore, we may assume
that $J$ is finite. Now, since each $\mathcal F_i$ is an \'{e}tale sheaf, we have:
\begin{equation}\label{215br}
\mathcal F_i(X)=lim(\prod_{j\in J}\mathcal F_i(X_j)\overset{\longrightarrow}{\underset{
\longrightarrow}{ }}\prod_{j,j'\in J}\mathcal F_i(X_j\times_XX_{j'}))
\end{equation} for each $i\in I$. Since $J$ is finite, the products appearing in \eqref{215br} are finite.
Now, since filtered colimits commute with finite limits and finite products in the category
of abelian groups, it follows from \eqref{215br} that
\begin{equation}
colim_{i\in I}\mathcal F_i(X)=lim(\prod_{j\in J} colim_{i\in I}\mathcal F_i(X_j)\overset{\longrightarrow}{\underset{
\longrightarrow}{ }}\prod_{j,j'\in J}colim_{i\in I}\mathcal F_i(X_j\times_XX_{j'}))
\end{equation} The result of \eqref{214br} is now clear from the definition
of $\mathcal F$ in \eqref{213br}.
\end{proof}

\medskip
\begin{rem}\emph{The result analogous to Lemma \ref{0L5} also holds for the Nisnevich site: let $\{\mathcal F_i\}_{i\in I}$ be a filtered inductive system
of objects in $Presh(Cor_K)$. Suppose that each $\mathcal F_i$ is also a sheaf for the Nisnevich topology on 
$Sm/K$ and let $\mathcal F$ be the presheaf colimit as in \eqref{213br}. Then, $\mathcal F$ is  a sheaf for the Nisnevich topology
on $Sm/K$. The key step in the proof
of Lemma \ref{0L5} is the fact that every \'{e}tale covering of a scheme $X$ in $Sm/K$ has a finite subcover. In the Nisnevich
case, we can proceed as follows: suppose that $\{X_j\longrightarrow X\}_{j\in J}$ is a Nisnevich covering 
in $Sm/K$. Since $X$ is smooth and hence of finite type over the field $K$, $X$ is Noetherian. Hence, any Nisnevich
covering of $X$ has a finite subcover (see, for instance, \cite[Proposition A.1.4]{Hadian}). Then, the rest of the 
proof of Lemma \ref{0L5} goes through without change in the Nisnevich case.}
\end{rem}

\medskip
\begin{thm}\label{P3} Let $\mathcal X$ be an ind-scheme, i.e., an ind-object of the
category $Sm/K$. Then, $Cor(\mathcal X)\in Presh(Cor_K)$ is actually
a sheaf for the \'{e}tale topology on $Sm/K$, i.e., $Cor(\mathcal X)=Cor_{et}(\mathcal X)$.
\end{thm}

\begin{proof} We consider any ind-object $\mathcal X$ of the category $Sm/K$. Then, by definition,
we know that
\begin{equation}\label{kohli}
 Cor(\mathcal X)=\underset{(X,x)\in Dir(\mathcal X)}{colim}\textrm{ }C_X
\end{equation} Since $\mathcal X$ is an ind-scheme, we know from Definition \ref{D31pz} that
the inductive system $Dir(\mathcal X)$ is filtered. Additionally, we know that for any
$X\in Sm/K$, $C_X$ is actually an \'{e}tale sheaf on $Sm/K$ (see \cite[Lemma 6.2]{WMV}).
Applying Lemma \ref{0L5}, it now follows that the presheaf $Cor(\mathcal X)$ defined
by the filtered colimit  in \eqref{kohli} is a sheaf for the  \'{e}tale topology
on $Sm/K$. This proves the result. 
\end{proof}

\medskip

\medskip

\section{Isomorphisms of Chow groups and $\mathbf{Cor}$ on ind-schemes}

\medskip

\medskip

\medskip
Let $\mathcal X$ be an ind-scheme over $K$, i.e., $\mathcal X$ is a filtered colimit of representable
presheaves on $Sm/K$. 
In this section, we introduce and study the Chow groups of the ind-scheme $\mathcal X$. More generally,
for any presheaf $\mathcal F$ of sets on $Sm/K$, we define Chow groups $CH^p(\mathcal F)$, $p\in \mathbb Z$ by 
modifying the definition in \cite{AB6} and in the spirit of the bivariant Chow theory of
Fulton and MacPherson \cite{FM}. In particular, when $\mathcal F=h_X$, the representable presheaf
on $Sm/K$ corresponding to an object $X\in Sm/K$, we show that the usual Chow group
of $X$ can be recovered as:
\begin{equation}\label{KSinha}
CH^p(h_X)\cong CH^p(X) \qquad\forall \textrm{ }p\in \mathbb Z
\end{equation} However, when the scheme $X\in Sm/K$ is considered as an object of 
category $Cor_K$, the representable presheaf on $Cor_K$ corresponding to $X$ is
the presheaf with transfers $C_X$:
\begin{equation}
C_X:Cor_K\longrightarrow \mathbf{Ab}\qquad Y\mapsto Cor(Y,X)\qquad \forall\textrm{ }Y\in Cor_K
\end{equation} Therefore, we would like to describe the Chow group of $X$ also in terms of the ``Chow group'' 
of the representable presheaf with transfers $C_X$. In order to do this, we introduce in Definition 
\ref{D32} the notion of Chow groups $\mathcal{CH}^p(\mathcal G)$, $p\in \mathbb Z$ for any
presheaf with transfers $\mathcal G$. Once again, our definition of the groups 
$\mathcal{CH}^p(\mathcal G)$ is in the spirit of the bivariant Chow theory of Fulton and 
MacPherson. Then, our first main result in this section is the following: for any
$X\in Sm/K$, we have  isomorphisms:
\begin{equation}\label{KSinha2}
CH^p(X)\cong CH^p(h_X)\cong \mathcal{CH}^p(C_X)\qquad\forall \textrm{ }p\in \mathbb Z
\end{equation} Thereafter, we proceed   to prove the analogue of \eqref{KSinha2} for an ind-scheme 
$\mathcal X$. Accordingly, we consider the Chow group of the presheaf with transfers 
$Cor(\mathcal X)$ associated to $\mathcal X$ and show that we have isomorphisms:
\begin{equation}\label{kitty}
CH^p(\mathcal X)\cong \mathcal{CH}^p(Cor(\mathcal X))\qquad \forall \textrm{ }p\in \mathbb Z
\end{equation} where $CH^p(\mathcal X)$ is the Chow group of $\mathcal X$ considered
as a presheaf of sets on $Sm/K$. 
Later on, in Section 4, we will show that the methods in this section
can be used to extend the filtration of Saito \cite{Saito} to the Chow groups of an arbitrary presheaf
$\mathcal F$ on $Sm/K$ as well as to the Chow groups of an arbitrary presheaf with transfers $\mathcal G$. In particular, for an 
ind-scheme $\mathcal X$, this gives us a filtration on $CH^p(\mathcal X)$ and on $\mathcal{CH}^p(Cor(\mathcal X))$ respectively. 
If the ind-scheme $\mathcal X$ is also ind-proper, we will show that the isomorphism $CH^p(\mathcal X)\cong \mathcal{CH}^p(Cor(\mathcal X))$ in \eqref{kitty} is actually
a filtered isomorphism with respect to these filtrations. 

\medskip
We start by defining the Chow groups of a presheaf $\mathcal F$ of sets on $Sm/K$. Given a morphism
$f:X\longrightarrow Y$ of schemes, Fulton and MacPherson \cite{FM} have defined the bivariant
Chow groups $CH^p(f:X\longrightarrow Y)$, $\forall$ $p\in \mathbb Z$. In particular, when
$X=Y$ and $f=1_X$ is the identity on $X$, we have the operational Chow groups $CH^p(1_X:X\longrightarrow X)$
of $X$. For an exposition on this theory, we refer the reader to \cite[$\S$ 17]{Fulton}. We will now
adapt their definition to obtain ``Chow groups'' for a presheaf $\mathcal F$ of sets on
$Sm/K$.

\medskip
\begin{defn}\label{D31}
 Let $\mathcal F$ be a presheaf of sets on $Sm/K$. For any scheme $X\in Sm/K$, let
$h_X$ be the presheaf on $Sm/K$ represented by $X$. Then, for any $p\in\mathbb Z$, a class
$c\in CH^p(\mathcal F)$ is that which associates to each morphism $x:h_X\longrightarrow \mathcal F$ of presheaves on $Sm/K$, a family of homomorphisms
\begin{equation}c_x^k:CH^k(X) \longrightarrow CH^{k+p}(X)\qquad \forall\textrm{ }k\in\mathbb Z\end{equation} such that the induced morphism $c_x=\underset{k\in \mathbb Z}{
\bigoplus}c_x^k:CH^*(X)=\underset{k\in \mathbb Z}{\bigoplus}CH^k(X)
\longrightarrow CH^*(X)=\underset{k\in \mathbb Z}{\bigoplus}CH^k(X)$
satisfies the following properties:

\medskip
\noindent (a) (well behaved with respect to flat pullbacks) Let $f:Y\longrightarrow X$ be a flat  morphism in $Sm/K$. Then, for any $\alpha\in CH^*(X)$, we have $
f^*(c_x(\alpha))=c_{x\circ f}(f^*(\alpha))$.

\medskip
\noindent (b) (well behaved with respect to proper pushforwards) Let $f:Y\longrightarrow X$ be a proper morphism in $Sm/K$. Then, for any $\alpha\in CH^*(Y)$,  we have
$f_*(c_{x\circ f}(\alpha))=c_{x}(f_*(\alpha))$.

\medskip
\noindent (c) (well behaved with respect to refined Gysin morphisms) Let $i:Y'\longrightarrow X'$ be a  regular imbedding in $Sm/K$  
and $g:X\longrightarrow X'$ a morphism in $Sm/K$ forming 
a fibre square
\begin{equation}\begin{CD}
Y @>g'>> Y' \\
@Vi'VV @ViVV \\
X @>g>> X' \\
\end{CD}
\end{equation} Then, if $Y$ lies in $Sm/K$, for any $\alpha\in CH^*(X)$, we have $
i^!(c_x(\alpha))=c_{x\circ i'}(i^!(\alpha))$. 

\medskip
 For the sake of convenience, when the morphism $x:h_X\longrightarrow \mathcal F$
is clear from context, we will often omit the subscript and write the morphism
$c_x$ simply as $c$.  
\end{defn}

\medskip
\begin{rem}\label{R3.2} \emph{Since we are dealing with smooth schemes, every morphism
$f:Y\longrightarrow X$ in $Sm/K$ can be factorized (see \cite[Appendix B.7.6]{Fulton}) as $f=p_X\circ i_f$, where $i_f:=(1\times f):
Y\longrightarrow Y\times X$ is a regular imbedding and $p_X:Y\times X\longrightarrow
X$ is the coordinate projection (and hence a flat morphism). Accordingly, we may write
\begin{equation} \label{4.3pab} f^*=i_f^!\circ p_X^*:CH^*(X)\longrightarrow CH^*(Y)
\end{equation} From (a) and
(c) in  Definition \ref{D31}, we know that, for any presheaf $\mathcal F$ on $Sm/K$, the morphisms induced by the classes 
$c\in CH^p(\mathcal F)$ behave well with respect to flat morphisms and regular imbeddings.
Hence, the expression in \eqref{4.3pab} tells us that we may as well remove the requirement of flatness from condition (a) in Definition \ref{D31} and
replace it with the following
condition:}

\emph{(a') Let $f:Y\longrightarrow X$ be a morphism  in $Sm/K$. Then, for any $\alpha\in CH^*(X)$, we have $f^*(c_x(\alpha))=c_{x\circ f}(f^*(\alpha))$.}
\end{rem}

\medskip
From Definition \ref{D31}, we note that if $f:\mathcal G\longrightarrow \mathcal F$
is a morphism of presheaves on $Sm/K$, we have pullbacks $f^*:CH^p(\mathcal F)
\longrightarrow CH^p(\mathcal G)$ defined as follows: for any 
$c\in CH^p(\mathcal F)$, we define $f^*(c)\in CH^p(\mathcal G)$ by setting
\begin{equation}\label{ATCR}
f^*(c)^k_x:=c^k_{f\circ x}:CH^k(X)\longrightarrow CH^{k+p}(X)
\qquad \forall\textrm{ }k\in \mathbb Z
\end{equation} for any morphism $x:h_X\longrightarrow \mathcal G$ with
$X\in Sm/K$. As mentioned before, the Chow groups $CH^p(\mathcal F)$ in Definition \ref{D31} have been
obtained by adapting the definition of the operational Chow group of Fulton and MacPherson \cite{FM} to presheaves. However,
it must be mentioned that when $\mathcal F=h_X$ for some $X\in Sm/K$, the definition of $CH^p(h_X)$ in Definition \ref{D31} differs slightly from the 
original definition of $CH^p(1_X:X\longrightarrow X)$ given by Fulton and MacPherson. In Definition \ref{D31}, the Chow group $CH^p(h_X)$ of the presheaf $h_X$ is 
defined by using morphisms $y:h_Y\longrightarrow h_X$ where $Y$ varies over all schemes in $Sm/K$. However, in the original definition of the operational Chow
group of $X$ in \cite{FM}, one uses morphisms $y:Y\longrightarrow X$ where $Y$ varies over all schemes, i.e., not necessarily  only those in $Sm/K$.  As such, we need to show that for any $X\in Sm/K$, the  Chow group 
of $X$ can be recovered
from Definition \ref{D31} by setting $\mathcal F=h_X$. 

\begin{thm}\label{MSDRNC} Let $X \in Sm/K$ and let $h_X$ denote the presheaf
on $Sm/K$ represented by $X$. Then, we have natural isomorphisms:
\begin{equation}
CH^p(h_X)\cong CH^p(X) \qquad \forall\textrm{ }p\in \mathbb Z
\end{equation} 
\end{thm}

\begin{proof} We fix some $p\in \mathbb Z$ and choose any $\alpha\in CH^p(X)$. Then, 
we can define a morphism $S^p:CH^p(X)\longrightarrow CH^p(h_X)$ as follows: given any
morphism $y:h_Y\longrightarrow h_X$ of representable presheaves (corresponding to a morphism
$y:Y\longrightarrow X$ in $Sm/K$), we define the class $S^p(\alpha)\in CH^p(h_X)$
by setting:
\begin{equation}\label{e34}
(S^p(\alpha))^k_y:CH^k(Y)\longrightarrow CH^{k+p}(Y) \qquad 
\beta\mapsto y^*(\alpha)\cdot \beta 
\end{equation} Conversely, we can define a morphism $T^p:CH^p(h_X)
\longrightarrow CH^p(X)$ as follows: for any class $c\in CH^p(h_X)$, we consider
the identity morphism $1:h_X\longrightarrow h_X$ and set:
\begin{equation}\label{e35}
T^p(c):=c_1^0([X])\in CH^p(X) 
\end{equation} From \eqref{e34} and \eqref{e35}, it is clear that $T^p\circ S^p$ is the
identity. 

\medskip
On the other hand, let us choose some $c\in CH^p(h_X)$ and consider a morphism $y:h_Y
\longrightarrow h_X$ of representable presheaves corresponding to a morphism
$y:Y\longrightarrow X$ in $Sm/K$. For some given $k\in \mathbb Z$, we consider
$\gamma\in CH^k(Y)$. In order to show that $S^p\circ T^p$ is the identity, we need to verify
that $c^k_y(\gamma)=y^*(T^p(c))\cdot \gamma$. It suffices to show this for the case
when $\gamma$ corresponds to an irreducible closed subscheme $V$ of codimension
$k$ in $Y$. Let $i:V\longrightarrow Y$ denote the closed immersion of $V$ into $Y$. 
Since $K$ is a field of characteristic zero, it admits a resolution of singularities and hence
we can choose  a proper birational morphism $\tilde{p}:\tilde{V}\longrightarrow V$ such
that $\tilde{V}$ is smooth. Since $\tilde{p}$ and
the closed immersion $i$ are both proper, the composition
$i\circ \tilde{p}:\tilde{V}\longrightarrow Y$  is proper (and hence separated). Combining
this with the fact that 
$Y$ is separated (since $Y\in Sm/K)$, it follows that $\tilde{V}$ is also separated. Hence,
$\tilde{V}\in Sm/K$.  We now have:
\begin{equation}
\begin{array}{ll}
c^k_y(\gamma)=c^k_y([V]) & = c^k_y((i\circ \tilde{p})_*([\tilde{V}]))
 =(i\circ \tilde{p})_*( c^0_{y\circ i\circ \tilde{p}}([\tilde{V}]))\\
& = (i\circ \tilde{p})_*(c^0_{y\circ i\circ \tilde{p}}((i\circ \tilde{p})^*([Y])))
= (i\circ \tilde{p})_*((i\circ \tilde{p})^*(c^0_y([Y])))\\
& = (i\circ \tilde{p})_*((i\circ \tilde{p})^*(c^0_y([Y]))\cdot [\tilde{V}])\\
&=c^0_y([Y])\cdot (i\circ \tilde{p})_*([\tilde{V}])=c^0_y([Y])\cdot \gamma\\
&=c^0_y(y^*[X])\cdot \gamma = y^*(c^0_1([X]))\cdot \gamma =y^*(T^p(c))\cdot \gamma\\
\end{array}
\end{equation}
This proves the result. 
\end{proof}

\medskip
Let $X$, $Y$ be schemes in $Sm/K$ and let $p_X:Y\times X\longrightarrow X$  and $p_Y:Y\times X\longrightarrow
Y$  be the respective coordinate projections. Let  $W\in Cor(Y,X)$ be an elementary finite correspondence
from $Y$ to $X$.  It follows that
the correspondence $W$ defines pullback morphisms ($\forall$ $p\in \mathbb Z$):
\begin{equation}\label{3e1}
W^*:CH^p(X)\longrightarrow CH^p(Y) \qquad 
\alpha \mapsto p_{Y*}(W\cdot p_X^*(\alpha)) 
\end{equation} Here, we note that when  $X$ is proper, $p_Y:Y\times X\longrightarrow Y$ is proper and the $p_{Y*}$ appearing
in \eqref{3e1} is simply the proper pushforward along $p_{Y}$.  For general $X$, one needs
to observe that $W\cdot p_X^*(\alpha)$ has finite support over $Y$ and hence
the pushforward $p_{Y*}(W\cdot p_X^*(\alpha))$ is still defined in $CH^p(Y)$ (see \cite[Example 2.5]{WMV}). Then, 
\eqref{3e1} gives us morphisms $W^*:CH^p(X)\longrightarrow CH^p(Y)$ for each
$W\in Cor(Y,X)$ and each $p\in \mathbb Z$. When the correspondence $W$
is given by the graph of a  morphism $f:Y\longrightarrow X$ in $Sm/K$, the pullback
morphisms in \eqref{3e1} are identical to the usual pullback on Chow cohomology groups.
We are now ready to define Chow groups for presheaves with transfers.

\medskip

\begin{defn}\label{D32}  
 Let $\mathcal G$ be a presheaf with transfers, i.e., $\mathcal G\in 
Presh(Cor_K)$. For any scheme $X\in Sm/K$, let
$h_X$ be the presheaf on $Sm/K$ represented by $X$.  Then, for any $p\in\mathbb Z$, a class
$c\in \mathcal{CH}^p(\mathcal G)$ is that which associates to each morphism $x:h_X\longrightarrow \mathcal G$ of presheaves on $Sm/K$, a family of homomorphisms
\begin{equation}c_x^k:CH^k(X) \longrightarrow CH^{k+p}(X)\qquad \forall\textrm{ }k\in\mathbb Z\end{equation} such that the induced morphism $c_x=\underset{k\in \mathbb Z}{
\bigoplus}c_x^k:CH^*(X)=\underset{k\in \mathbb Z}{\bigoplus}CH^k(X)
\longrightarrow CH^*(X)=\underset{k\in \mathbb Z}{\bigoplus}CH^k(X)$
satisfies the following properties:

\medskip
\noindent (a) Let $W\in Cor(Y,X)$ be a correspondence from $Y$ to $X$, inducing a morphism
$\mathcal G(W):\mathcal G(X)\longrightarrow \mathcal G(Y)$. The morphism $x:h_X
\longrightarrow \mathcal G$ corresponds to an element $x\in \mathcal G(X)$ and we let $y=\mathcal G(W)(x)
\in \mathcal G(Y)$. Now, the element $y\in \mathcal G(Y)$ corresponds to a morphism
$y:h_Y\longrightarrow \mathcal G$ of presheaves on $Sm/K$. 
Then, for any $\alpha\in CH^*(X)$,  we have 
\begin{equation}
W^*(c_x(\alpha))=c_{y}(W^*(\alpha))
\end{equation} In particular, if $W=\Gamma_f$, the graph
of a morphism $f:Y\longrightarrow X$ in $Sm/K$, we have 
\begin{equation}f^*(c_x(\alpha))
=c_{x\circ f}(f^*(\alpha))
\end{equation} for any $\alpha\in CH^*(X)$. 

\medskip
\noindent (b) Let $f:Y\longrightarrow X$ be a proper morphism in $Sm/K$. Then, for any $\alpha\in CH^*(Y)$, we have
$f_*(c_{x\circ f}(\alpha))=c_{x}(f_*(\alpha))$.

\medskip
\noindent (c) Let $i:Y'\longrightarrow X'$ be a regular imbedding in $Sm/K$  
and $g:X\longrightarrow X'$ a morphism in $Sm/K$ forming 
a fibre square
\begin{equation}\begin{CD}
Y @>g'>> Y' \\
@Vi'VV @ViVV \\
X @>g>> X' \\
\end{CD}
\end{equation} Then, if $Y$ lies in $Sm/K$, for any $\alpha\in CH^*(X)$, we have $
i^!(c_x(\alpha))=c_{x\circ i'}(i^!(\alpha))$. 

\medskip
 For the sake of convenience, when the morphism $x:h_X\longrightarrow \mathcal G$
is clear from context, we will often omit the subscript and write the morphism
$c_x$ simply as $c$.  
\end{defn}

\medskip
\begin{thm}\label{T3.5} Let $X\in Sm/K$ and let $h_X$ be the presheaf on $Sm/K$ 
represented by $X$. Let $C_X$ be the presheaf with transfers defined
by setting $C_X(Y):=Cor(Y,X)$ for any $Y\in Cor_K$. Then, we have natural isomorphisms
\begin{equation}\label{Z3.18}
CH^p(X)\cong CH^p(h_X)\cong \mathcal{CH}^p(C_X) \qquad \forall \textrm{ }p
\in \mathbb Z
\end{equation}
\end{thm}

\begin{proof} Since $X\in Sm/K$, we already know from Proposition \ref{MSDRNC} that
$CH^p(X)\cong CH^p(h_X)$. Let $c\in \mathcal{CH}^p(C_X)$. We define a morphism $S^p:\mathcal{CH}^p
(C_X)\longrightarrow CH^p(h_X)$ as follows: for any morphism $f:Y\longrightarrow X$ inducing
a morphism $f:h_Y\longrightarrow h_X$, 
we define the class $S^p(c)\in CH^p(h_X)$ by setting:
\begin{equation}\label{namo}
(S^p(c))^k_f:=c^k_{\Delta\circ f}:CH^k(Y)\longrightarrow CH^{k+p}(Y) 
\end{equation} where $\Delta:h_X\longrightarrow C_X$ is the morphism of presheaves
induced by the diagonal correspondence $\Delta$ from $X$ to $X$. Conversely, choose 
$\alpha\in CH^p(X)\cong CH^p(h_X)$. We define a morphism 
$T^p:CH^p(h_X)\longrightarrow \mathcal{CH}^p(C_X)$ as follows: for any morphism
$y:h_Y\longrightarrow C_X$ of presheaves defining an element $y\in Cor(Y,X)=C_X(Y)$, we set:
\begin{equation}\label{lka}
(T^p(\alpha))^k_y:CH^k(Y)\longrightarrow CH^{k+p}(Y)\qquad 
\beta\mapsto y^*(\alpha)\cdot \beta
\end{equation} From these definitions and the proof of the isomorphism
$CH^p(X)\cong CH^p(h_X)$ in Proposition \ref{MSDRNC}, it is clear that
the composition $S^p\circ T^p$ is the identity. 

\medskip
Let us choose  a morphism $y:h_Y
\longrightarrow C_X$ of presheaves and $\gamma\in CH^l(Y)$ for some given 
$l\in \mathbb Z$. Then, $y$ may be seen as an element of $C_X(Y)$, i.e., a correspondence
from $Y$ to $X$. It follows from \eqref{namo} and \eqref{lka} and the description
of the isomorphism  $CH^p(X)\cong CH^p(h_X)$ in the proof of Proposition \ref{MSDRNC}
that
\begin{equation}\label{shivraj}
(T^p\circ S^p)(c)^l_y(\gamma)=y^*(c_\Delta^0([X]))\cdot \gamma
\end{equation} Since $C_X$ is a presheaf with transfers, the correspondence $y
\in Cor(Y,X)$ 
induces a morphism $C_X(y):C_X(X)\longrightarrow C_X(Y)$ and it is clear that
$C_X(y)(\Delta)=y$. Hence, from condition (a) in Definition \ref{D32}, it follows that
\begin{equation}\label{raman}
y^*(c_\Delta^0([X]))=c_y^0(y^*([X]))=c_y^0([Y]) \in CH^p(Y)
\end{equation} From \eqref{shivraj} and \eqref{raman}, it now follows that 
\begin{equation}\label{manohar}
(T^p\circ S^p)(c)^l_y(\gamma)=c_y^0([Y])\cdot \gamma
\end{equation} Now, for sake of simplicity, we may assume that $\gamma
\in CH^l(Y)$ corresponds to an irreducible closed subscheme $V$ of codimension $l$ admitting a closed
immersion $i:V\longrightarrow Y$. Since 
$K$ is a field of characteristic zero, using resolution of singularities, we can 
obtain a proper birational morphism $\tilde{p}:\tilde{V}\longrightarrow V$
such that $\tilde{V}$ is smooth. Then, as in the proof of Proposition 
\ref{MSDRNC}, it follows that $\tilde{V}\in Sm/K$. Then, 
the class $c\in \mathcal{CH}^p(C_X)$ induces morphisms 
\begin{equation}
c^l_y:CH^l(Y)\longrightarrow CH^{l+p}(Y)
\end{equation} and in particular, we have:
\begin{equation}\label{shettar}
\begin{array}{ll}
c^l_y(\gamma)= c^l_y([V]) & = c^l_y((i\circ \tilde{p})_*([\tilde{V}])) =(i\circ \tilde{p})_*
c^0_{y\circ i\circ \tilde{p}}([\tilde{V}])\\
& =(i\circ \tilde{p})_*
c^0_{y\circ i\circ \tilde{p}}((i\circ \tilde{p})^*([Y])) \\
& = (i\circ \tilde{p})_*((i\circ \tilde{p})^*(c_y^0([Y]))\cdot [\tilde{V}])\\
& =c_y^0([Y])\cdot  (i\circ \tilde{p})_*([\tilde{V}]) = c_y^0([Y])\cdot \gamma
\end{array}
\end{equation} From \eqref{manohar} and \eqref{shettar}, it follows that
$(T^p\circ S^p)(c)=c\in \mathcal{CH}^p(C_X)$. Hence, $CH^p(X)\cong CH^p(h_X)
\cong \mathcal{CH}^p(C_X)$, $\forall$ $p\in \mathbb Z$. 
\end{proof}

\medskip
\begin{lem}\label{sinha} (a) Let $\{\mathcal F_i\}_{i\in I}$ be a filtered inductive system of objects in $Presh(Sm/K)$. Let $\mathcal F:=colim_{i\in I}
\mathcal F_i$. Then, we have isomorphisms:
\begin{equation}
CH^p(\mathcal F)\cong lim_{i\in I}CH^p(\mathcal F_i) \qquad \forall \textrm{ }p\in 
\mathbb Z
\end{equation}

\medskip
(b) Let $\{\mathcal G_i\}_{i\in I}$ be a filtered inductive 
system of objects in $Sh_{et}(Cor_K)$. Let $\mathcal G:=colim_{i\in I}
\mathcal G_i$. Then, we have isomorphisms:
\begin{equation}
\mathcal{CH}^p(\mathcal G)\cong lim_{i\in I}\mathcal{CH}^p(\mathcal G_i) \qquad \forall \textrm{ }p\in 
\mathbb Z
\end{equation}

\end{lem}

\begin{proof}(a)  We fix any $p\in \mathbb Z$. Since $\mathcal F$ is the colimit of the system $\{\mathcal F_i\}_{i\in I}$, 
the canonical morphisms $f_i:\mathcal F_i\longrightarrow \mathcal F$, $i\in I$
induce pullbacks $f_i^*:CH^p(\mathcal F)\longrightarrow CH^p(\mathcal F_i)$ on 
Chow groups as described in \eqref{ATCR}. Now suppose that we have a collection of
classes $c_i\in CH^p(\mathcal F_i)$, $i\in I$ such that $f_{ij}^*(c_j)=c_i$ for any
morphism $f_{ij}:\mathcal F_i\longrightarrow \mathcal F_j$ in the system $I$. 
We will show that the collection $\{c_i\}_{i\in I}$ determines a unique element
$c\in CH^p(\mathcal F)$ such that $f_i^*(c)=c_i$ for each $i\in I$. 

\medskip
In order to define $c\in CH^p(\mathcal F)$, we proceed as follows: we consider a morphism
$x:h_X\longrightarrow \mathcal F$ (for some $X\in Sm/K$) of presheaves determining an element 
$x\in \mathcal F(X)$. By definition of $\mathcal F$, we know that:
\begin{equation}
\mathcal F(X)=colim_{i\in I}\mathcal F_i(X)
\end{equation} Hence, there exists some $i_0\in I$ and some $x_{i_0}\in 
\mathcal F_{i_0}(X)$ such that $f_{i_0}(X)(x_{i_0})=x$ where $f_{i_0}(X):
\mathcal F_{i_0}(X)\longrightarrow \mathcal F(X)$ is the morphism induced by
$f_{i_0}:\mathcal F_{i_0}\longrightarrow \mathcal F$. Then, the class $c\in 
CH^p(\mathcal F)$ is defined by setting:
\begin{equation}
c^k_x:= (c_{i_0})^k_{x_{i_0}}:CH^k(X)\longrightarrow CH^{k+p}(X) \qquad 
\forall \textrm{ }k\in \mathbb Z
\end{equation} We need to check that this class $c$  is well defined, i.e., given $i_0'\in I$ and 
some $x_{i_0'}\in \mathcal F_{i_0'}(X)$ such that $f_{i_0'}(X)(x_{i_0'})=x
\in \mathcal F(X)$, we should have $ (c_{i_0})^k_{x_{i_0}}=
 (c_{i_0'})^k_{x_{i_0'}}:CH^k(X)\longrightarrow CH^{k+p}(X)$, $\forall$ $k
\in \mathbb Z$. Since $I$ is filtered, there exists some $j_0\in I$
 and $x_{j_0}\in \mathcal F_{j_0}(X)$ along with morphisms 
$f_{i_0j_0}:\mathcal F_{i_0}\longrightarrow \mathcal F_{j_0}$, 
$f_{i_0'j_0}:\mathcal F_{i_0'}\longrightarrow \mathcal F_{j_0}$ in the system $I$ such that
$f_{i_0j_0}(X)(x_{i_0})=x_{j_0}=f_{i_0'j_0}(X)(x_{i_0'})$.  Then, 
we know that $f_{i_0j_0}^*(c_{j_0})=c_{i_0}\in CH^p(\mathcal F_{i_0})$
and $f_{i_0'j_0}^*(c_{j_0})=c_{i_0'}\in CH^p(\mathcal F_{i_0'})$. It follows that
\begin{equation}\label{rajnath}
c^k_x= (c_{i_0})^k_{x_{i_0}}=(f_{i_0j_0}^*(c_{j_0}))^k_{x_{i_0}}=(c_{j_0})^k_{x_{j_0}}
=(f_{i_0'j_0}^*(c_{j_0}))^k_{x_{i_0'}}=(c_{i_0'})^k_{x_{i_0'}}
\end{equation} Hence, we have a well defined
class $c\in CH^p(\mathcal F)$ determined by the collection $c_i\in CH^p(\mathcal F_i)$, 
$i\in I$. It follows that $CH^p(\mathcal F)$ is the limit of the system
$\{CH^p(\mathcal F_i)\}_{i\in I}$. 

\medskip
(b) Since $I$ is a filtered inductive system and each $\mathcal G_i\in Sh_{et}(Cor_K)$, 
it follows from Lemma \ref{0L5} that the colimit $\mathcal G\in Sh_{et}(Cor_K)$ is given by:
\begin{equation}\label{kaneenikas}
\mathcal G(X)=colim_{i\in I}\mathcal G_i(X)\qquad \forall\textrm{ }X\in Cor_K
\end{equation} The colimit in \eqref{kaneenikas} is a filtered colimit 
of abelian groups. Additionally, we know that
the underlying set of the filtered colimit of a system of abelian groups is identical
to the filtered colimit of the system of their underlying sets. Hence, the method of proving
part (a) can now be repeated in order to prove (b). 
\end{proof}

\medskip

\begin{thm}\label{HHH}  Let $\mathcal X\in Ind-(Sm/K)$ be an ind-scheme. Then, we have 
isomorphisms:
\begin{equation}\label{Devendra}
CH^p(\mathcal X)\cong \mathcal{CH}^p(Cor(\mathcal X)) \qquad \forall\textrm{ }
p\in \mathbb Z
\end{equation}
\end{thm}

\begin{proof} For any $X\in Sm/K$, let $h_X$ denote the presheaf on $Sm/K$ represented
by $X$. Since $\mathcal X$ is an ind-scheme, the inductive system $Dir(\mathcal X)$ 
as defined in \eqref{zqap} is filtered. We know that
\begin{equation}\label{ksinha}
\mathcal X=\underset{(X,x)\in Dir(\mathcal X)}{colim}\textrm{ } h_X 
\end{equation} From Lemma \ref{sinha}(a), it now follows that
\begin{equation}\label{ksinha1}
CH^p(\mathcal X)\cong\underset{(X,x)\in Dir(\mathcal X)}{lim}\textrm{ } CH^p(h_X)
\end{equation} For any $X\in Cor_K$ let $C_X$ denote the presheaf with transfers
defined by setting $C_X(Y)=Cor(Y,X)$ for any $Y\in Cor_K$. We have noted before that
$C_X$ is an \'{e}tale sheaf with transfers, i.e., $C_X\in Sh_{et}(Cor_K)$. Then, by definition, we know that
\begin{equation}
Cor(\mathcal X)=\underset{(X,x)\in Dir(\mathcal X)}{colim}\textrm{ } C_X
\end{equation} Again, since each $C_X\in Sh_{et}(Cor_K)$ and $Dir(\mathcal X)$ is filtered, it follows from Lemma \ref{sinha}(b) that
\begin{equation}\label{ksinha2}
\mathcal{CH}^p(Cor(\mathcal X))\cong\underset{(X,x)\in Dir(\mathcal X)}{lim}\textrm{ }
\mathcal{CH}^p(C_X)
\end{equation} Finally, from Proposition \ref{T3.5} we know that
$CH^p(h_X)\cong \mathcal{CH}^p(C_X)$ for each $(X,x)\in 
Dir(\mathcal X)$. Combining this with \eqref{ksinha1} and \eqref{ksinha2}, it follows
that $CH^p(\mathcal X)\cong \mathcal{CH}^p(Cor(\mathcal X))$. 

\end{proof}

\medskip

\medskip

\section{Extensions of Saito's filtration and filtered isomorphisms of Chow groups} 

\medskip

\medskip
In this section, we let the ground field $K=\mathbb C$. For any scheme $Y\in SmProj/\mathbb C$, the full subcategory of $Sm/\mathbb C$ consisting of schemes over $\mathbb C$
that are smooth as well as projective, Saito \cite{Saito} has defined a  filtration 
$\{F^rCH^p(Y)\}_{r\geq 0}$ on the Chow groups of $Y$. In this section, we start by showing that we can extend Saito's filtration to the Chow groups of  any presheaf $\mathcal F$ on $Sm/\mathbb C$. In other words,  we will construct
a descending filtration:
\begin{equation}
CH^p(\mathcal F)=F^0CH^p(\mathcal F) \supseteq F^1CH^p(\mathcal F) \supseteq F^2CH^p(\mathcal F)
\supseteq \dots 
\end{equation} on $CH^p(\mathcal F)$ that generalizes Saito's filtration. For any $X\in Sm/\mathbb C$, we know from Proposition 
\ref{MSDRNC} that $CH^p(h_X)\cong CH^p(X)$ $\forall$ $p\geq 0$, where $h_X$ is the presheaf 
on $Sm/\mathbb C$ represented
by $X$. 
Then, in particular, if we set $\mathcal F=h_X$ for some $X\in Sm/\mathbb C$, 
we obtain a filtration on the Chow groups $CH^p(h_X)\cong CH^p(X)$. Thus, we have an extension
of Saito's filtration to the Chow groups of any $X\in Sm/\mathbb C$, i.e., to the Chow groups of any
smooth and separated (but not necessarily projective) scheme over $\mathbb C$. On the other hand, if we choose 
the presheaf to be an ind-scheme $\mathcal X$ over $\mathbb C$, we have a filtration on the Chow groups
$CH^p(\mathcal X)$. 

\medskip
Further, in this section, we define an extension of Saito's filtration to the Chow groups
$\mathcal{CH}^p(\mathcal G)$ of an arbitrary presheaf with transfers $\mathcal G$. In particular, this
gives us filtrations on the Chow groups $\mathcal{CH}^p(C_X)$ for any scheme $X\in Sm/\mathbb C$ 
and  the Chow groups $\mathcal{CH}^p(Cor(\mathcal X))$ for any ind-scheme $\mathcal X\in Ind-(Sm/\mathbb C)$. Then, if
the scheme $X\in Sm/\mathbb C$ is also proper, we show that the isomorphism $CH^p(h_X)\cong \mathcal{CH}^p(C_X)$
from Proposition \ref{T3.5} is actually a filtered isomorphism. Similarly, if the ind-scheme $\mathcal X$ is also ind-proper, we show that
the isomorphism $CH^p(\mathcal X)\cong \mathcal{CH}^p(Cor(\mathcal X))$ from Proposition \ref{HHH} is actually a filtered
isomorphism.

\medskip
We recall here that Saito's filtration \cite{Saito} on the Chow groups of a smooth projective
scheme $Y$ is a candidate for the conjectural
Bloch-Beilinson filtration on Chow groups (see, for instance, Levine \cite[$\S$ 7]{Levine2}). 
In particular, Saito's filtration is well behaved with respect to pullbacks and pushforwards, 
and more generally, with respect to the action of any algebraic correspondence (see 
\cite[Lemma 2.11]{Saito}). Moreover, the filtration is well behaved with respect
to intersection products (see \cite[$\S$ 7.3.2-7.3.3]{Levine2}), i.e., the intersection
product  $CH^p(Y)\otimes CH^q(Y)\longrightarrow CH^{p+q}(Y)$ on the Chow groups of $Y$ induces products:
\begin{equation}\label{ee6.2}
F^rCH^p(Y)\otimes F^sCH^q(Y)\longrightarrow F^{r+s}CH^{p+q}(Y)
\qquad \forall\textrm{ }Y\in SmProj/\mathbb C, \textrm{ }p,q,r,s\geq 0
\end{equation} For more on the conjectures of Beilinson and their connections
to mixed motives, we refer the reader to Beilinson \cite{Beil},  
Jannsen \cite{Jan} and Levine 
\cite{Levine2}. 

\medskip
We are now ready to recall the filtration of Saito \cite[Definition 2.10]{Saito}. As in 
\cite[$\S$ 1]{Saito}, we fix a cohomology theory in the sense of Bloch and Ogus \cite{BO}. 
In other words, for every closed immersion $Z\hookrightarrow X$ of schemes,
we are given a graded abelian group 
$H_Z^*(X)=\underset{i\in \mathbb Z}{\bigoplus}H_Z^i(X)$ satisfying certain compatibility
conditions described in \cite{BO}. When the closed immersion is 
the identity $1:X\longrightarrow X$, $H_X^i(X)$ is denoted simply by $H^i(X)$
for each $i\in \mathbb Z$. Given such a cohomology theory, recall that we have the coniveau
filtration on each $H^i(X)$ defined as follows (for details  see, for instance, \cite[(1.2)]{Saito}):
\begin{equation}
N^tH^i(X):=\underset{Z\hookrightarrow X,codim(Z)\geq t}{\sum}Im(H^i_Z(X)
\longrightarrow H^i(X))\qquad
\forall\textrm{ }t\geq 0
\end{equation}

\medskip
\begin{defn} (see \cite[Definition 2.10]{Saito}) Let $Y\in SmProj/\mathbb C$. For any
$p\geq 0$, we
define a descending filtration 
\begin{equation}\label{e6.4}
CH^p(Y)=F^0CH^p(Y) \supseteq F^1CH^p(Y) \supseteq F^2CH^p(Y)
\supseteq \dots 
\end{equation} on  the Chow group $CH^p(Y)$ in the following inductive manner:

\medskip
(1) We set $F^0CH^p(Y):=CH^p(Y)$, for all $p\geq 0$. 

\medskip
(2) For some given $r\geq 0$, suppose that we have already defined $F^rCH^p(Y)$
for every $Y\in SmProj/\mathbb C$. Then, we define:
\begin{equation}
F^{r+1}CH^p(Y):=\underset{V,q,\Gamma}{\sum}
Im(\Gamma_*:F^rCH^{p+d_V-q}(V)\longrightarrow CH^p(Y))
\end{equation} where $V$, $q$ and $\Gamma$ range over the following data:

\medskip
(a) $V$  a smooth, projective variety of dimension $d_V$ over $\mathbb C$. 

(b) An integer $q$ such that $p\leq q\leq p+d_V$. 

(c) $\Gamma$ is a cycle of codimension $q$ in $V\times Y$ satisfying the condition
\begin{equation}\label{6.3}
\gamma^{2p-r}\in H^{2q-2p+r}(V)\otimes N^{p-r+1}H^{2p-r}(Y) 
\end{equation} where $\gamma=cl^q(V\times Y)$ is the cohomology class of
$\Gamma$ in $H^{2q}(V\times Y)$ and $\gamma^{2p-r}$ is the K\"{u}nneth
component of $\gamma\in H^{2q}(V\times Y)$ lying in 
$H^{2q-2p+r}(V)\otimes H^{2p-r}(Y)$.

\end{defn}

We are now ready to define the filtration $\{F^rCH^p(\mathcal F)\}_{
r\geq 0}$ on the Chow groups of a presheaf $\mathcal F$
on $Sm/\mathbb C$. 

\begin{defn}\label{D62} Let $\mathcal F$ be a presheaf of sets on $Sm/\mathbb C$ and let $c\in CH^p(
\mathcal F)$ for some $p\geq 0$.  For any $X\in Sm/\mathbb C$, let $h_X$ be the presheaf
on $Sm/\mathbb C$ represented by $X$. For any $r\geq 0$, we say that  $c\in F^rCH^p(\mathcal F)\subseteq 
CH^p(\mathcal F)$ if for any morphism $y:h_Y\longrightarrow \mathcal F$ with $Y
\in SmProj/\mathbb C$, the induced
morphisms
\begin{equation}
c^k_y: CH^k(Y)\longrightarrow CH^{k+p}(Y)\qquad \forall \textrm{ }k\in \mathbb Z
\end{equation} each satisfy
\begin{equation} 
Im(c^k_y|F^sCH^k(Y))\subseteq F^{s+r}CH^{k+p}(Y)\qquad \forall
\textrm{ }s\geq 0
\end{equation} where $\{F^sCH^k(Y)\}_{s\geq 0}$, $k\geq 0$ is Saito's filtration
on the Chow groups of $Y\in SmProj/\mathbb C$.   
\end{defn}

 From Definition \ref{D62}, it is clear that $F^rCH^p(\mathcal F)\supseteq 
F^{r+1}CH^p(\mathcal F)$ for any $r\geq 0$, i.e., $\{F^rCH^p(\mathcal F)\}_{r\geq 0}$
determines a descending filtration on $CH^p(\mathcal F)$. 
We can now extend Saito's filtration to the Chow groups of any smooth and separated (but not necessarily projective) scheme $X$ over
$\mathbb C$.
For this we suppose that $\mathcal F=h_X$ for some scheme $X\in Sm/\mathbb C$. Then,
Definition \ref{D62} gives us a filtration $\{F^rCH^p(h_X)\}_{r\geq 0}$ on the Chow groups
of $h_X$. From Proposition \ref{MSDRNC}, we know  that $CH^p(h_X)\cong CH^p(X)$ and  it follows that we have an induced  filtration:
\begin{equation}\label{e6.9}
CH^p(X)\cong CH^p(h_X)=F^0CH^p(h_X)\supseteq F^1CH^p(h_X)
\supseteq F^2CH^p(h_X)\supseteq \dots
\end{equation} on $CH^p(X)$ for any smooth and separated (but not necessarily projective) scheme $X$ over $\mathbb C$. Our first aim is to show that when $X$ is in fact smooth and projective, i.e.
$X\in SmProj/\mathbb C$, the filtration in \eqref{e6.9}
recovers Saito's filtration on $CH^p(X)$. 

\medskip
\begin{thm}\label{AmitShah} Let $Y\in SmProj/\mathbb C$ and let $h_Y$ be the presheaf on $Sm/\mathbb C$
represented by $Y$. We choose any $p\geq 0$. Let $\{F^rCH^p(Y)\}_{r\geq 0}$ be Saito's filtration on
$CH^p(Y)$ as defined in \eqref{e6.4}.  Let $\{F^rCH^p(h_Y)\}_{r\geq 0}$
be the descending filtration on $CH^p(h_Y)$ as defined in \eqref{e6.9}. Then, the isomorphism $CH^p(h_Y)
\cong CH^p(Y)$ is well behaved with respect to these filtrations, i.e.,  
we have induced isomorphisms:
\begin{equation}
F^rCH^p(h_Y)\cong F^rCH^p(Y)\qquad \forall\textrm{ }r\geq 0
\end{equation} 
\end{thm}

\begin{proof} We know that $Y\in SmProj/\mathbb C$ and hence $Y\in Sm/\mathbb C$. 
Choose $p\geq 0$. From the proof of Proposition \ref{MSDRNC}, we 
recall the mutually inverse morphisms $S^p:CH^p(Y)\longrightarrow CH^p(h_Y)$ and 
$T^p:CH^p(h_Y)\longrightarrow CH^p(Y)$ defining the isomorphism 
$CH^p(Y)\cong CH^p(h_Y)$.  

\medskip
We consider some $\alpha\in F^rCH^p(Y)$. Then, for any $z:h_Z\longrightarrow h_X$
with $Z\in Sm/\mathbb C$, the class $S^p(\alpha)\in CH^p(h_Y)$ determines morphisms:
\begin{equation}
(S^p(\alpha))_z^k:CH^k(Z)\longrightarrow CH^{k+p}(Z) \qquad \beta
\mapsto z^*(\alpha)\cdot \beta
\end{equation} In particular, suppose that $Z\in SmProj/\mathbb C$. Since Saito's filtration
on $CH^p(Y)$ is respected by pullbacks (see \cite[Lemma 2.11]{Saito})and $\alpha
\in F^rCH^p(Y)$, it follows that $z^*(\alpha)\in F^rCH^p(Z)$. Further, since Saito's filtration
on $CH^p(Z)$ is well behaved with respect to intersection products (as described in \eqref{ee6.2}),
it follows that for any $\beta\in F^sCH^k(Y)$, we have 
\begin{equation}
(S^p(\alpha))_z^k(\beta)=z^*(\alpha)\cdot \beta \in F^{r+s}CH^{k+p}(Z)
\end{equation} and hence $S^p(\alpha)\in F^rCH^p(h_Y)$ for any $\alpha\in F^rCH^p(Y)$. 

\medskip
Conversely, consider a class $c\in F^rCH^p(h_Y)$. Then, we have
\begin{equation}
T^p:CH^p(h_Y)
\longrightarrow CH^p(Y) \qquad c\mapsto c_1^0([Y])
\end{equation} Since $[Y]\in F^0CH^0(Y)$, $c\in F^rCH^p(h_Y)$  and $Y\in SmProj/\mathbb C$,
it follows from  Definition \ref{D62} that $c_1^0([Y])\in F^rCH^p(Y)$. Hence, $T^p(c)
=c_1^0([Y])
\in F^rCH^p(Y)$ and it now follows that $F^rCH^p(Y)\cong F^rCH^p(h_Y)$. 

\end{proof}

\medskip

\begin{thm}\label{P64} (a) Let $\mathcal F$ be a presheaf of sets on $Sm/\mathbb C$. Then, the filtration $\{F^rCH^p(\mathcal F)\}_{r\geq 0}$
on the Chow groups of $\mathcal F$ is well behaved with respect to products, i.e., we
have products:
\begin{equation}
F^rCH^p(\mathcal F)\otimes F^sCH^q(\mathcal F)\longrightarrow 
F^{r+s}CH^{p+q}(\mathcal F) \qquad \forall\textrm{ }p,q,r,s\geq 0
\end{equation}

(b) Let $f:\mathcal G\longrightarrow \mathcal F$ be a morphism of presheaves of sets 
on $Sm/\mathbb C$. Then, the filtration $\{F^rCH^p(\mathcal F)\}_{r\geq 0}$ on the Chow groups of $\mathcal F$
is well behaved with respect to pullbacks, i.e., the pullbacks restrict to morphisms:
\begin{equation}
f^*:F^rCH^p(\mathcal F)\longrightarrow F^rCH^p(\mathcal G)\qquad \forall \textrm{ }p,r\geq 0
\end{equation} 
\end{thm}

\begin{proof} (a) We consider any $Y\in SmProj/\mathbb C$ along with a morphism
$y:h_Y\longrightarrow \mathcal F$. We choose classes $c\in F^rCH^p(\mathcal F)$
and $d\in F^sCH^q(\mathcal F)$. Then, for any $\alpha\in F^tCH^k(Y)$, it follows from
Definition \ref{D62} that $d_y^k(\alpha)\in F^{s+t}CH^{k+q}(Y)$. Further, since
$c\in F^rCH^p(\mathcal F)$, we have
\begin{equation}
(c\cdot d)_y^k(\alpha)=c_y^{k+q}(d_y^k(\alpha))\in F^{r+s+t}CH^{k+p+q}(Y)
\end{equation} Hence, the class $(c\cdot d)\in F^{r+s}CH^{p+q}(\mathcal F)$. 

\medskip
(b) We consider any $Y\in SmProj/\mathbb C$ along with a morphism
$y:h_Y\longrightarrow \mathcal G$.  Given $p$, $r\geq 0$, we choose some $c\in F^rCH^p(\mathcal F)$ and consider the class $f^*(c)\in CH^p(\mathcal G)$. Let
$\alpha\in F^tCH^k(Y)$ for some $k$, $t\geq 0$. Then, 
\begin{equation}
f^*(c)_y^k(\alpha)=c_{f\circ y}^k(\alpha)=F^{r+t}CH^{p+k}(Y)
\end{equation} This proves the result. 
\end{proof}

\medskip
In particular, from Proposition \ref{P64} above, it follows that for any smooth and separated (but not  necessarily
projective) scheme $X$ over $\mathbb C$, the filtration 
\begin{equation}\label{Z4.18}
CH^p(X)\cong CH^p(h_X)=F^0CH^p(h_X)\supseteq F^1CH^p(h_X)\supseteq F^2CH^p(h_X)
\supseteq \dots 
\end{equation} as in \eqref{e6.9} is compatible with pullbacks and intersection products. Further, by choosing
$\mathcal F$ to be an ind-scheme $\mathcal X$, it follows from Definition \ref{D62} 
that we have a filtration  (for each $p\geq 0$)
\begin{equation}\label{Z4.19}
CH^p(\mathcal X)=F^0CH^p(\mathcal X)\supseteq F^1CH^p(\mathcal X)\supseteq F^2CH^p(\mathcal X)
\supseteq \dots 
\end{equation} on the Chow groups of the ind-scheme. We now extend
Saito's filtration to the Chow groups of presheaves with transfers.

\medskip
\begin{defn}\label{D66} Let $\mathcal G$ be a presheaf with transfers, i.e., $\mathcal G\in Presh(Cor_{\mathbb C})$. 
For any $X\in Sm/\mathbb C$, let $h_X$ be the presheaf
on $Sm/\mathbb C$ represented by $X$. For any $r\geq 0$, we say that  $c\in F^r\mathcal{CH}^p(\mathcal G)\subseteq 
\mathcal{CH}^p(\mathcal G)$ if for any morphism $y:h_Y\longrightarrow \mathcal G$ with $Y
\in SmProj/\mathbb C$, the induced
morphisms
\begin{equation}
c^k_y: CH^k(Y)\longrightarrow CH^{k+p}(Y)\qquad \forall \textrm{ }k\in \mathbb Z
\end{equation} each satisfy
\begin{equation} \label{Z4.21}
Im(c^k_y|F^sCH^k(Y))\subseteq F^{s+r}CH^{k+p}(Y)\qquad \forall
\textrm{ }s\geq 0
\end{equation} where $\{F^sCH^k(Y)\}_{s\geq 0}$, $k\geq 0$ is Saito's filtration
on the Chow groups of $Y\in SmProj/\mathbb C$.   
\end{defn} 

\medskip
In particular, it follows from Definition \ref{D66} that for any scheme $X\in Sm/\mathbb C$, we
have a filtration $\{F^r\mathcal{CH}^p(C_X)\}_{r\geq 0}$ on the Chow groups
of the presheaf with transfers $C_X$. We now show that for any scheme $X\in Sm/\mathbb C$ that
is also proper,
the filtration on $CH^p(X)\cong CH^p(h_X)$ in \eqref{Z4.18} corresponds to this filtration
on $\mathcal{CH}^p(C_X)$ via the isomorphism $CH^p(X)\cong CH^p(h_X)\cong \mathcal{CH}^p(C_X)$
in \eqref{Z3.18}. 

\medskip
\begin{thm}\label{P68} Let $X\in Sm/\mathbb C$ be a smooth and separated scheme over $\mathbb C$ that
is also proper. 
We choose any $p\geq 0$. Let $\{F^rCH^p(h_X)\}_{r\geq 0}$ be the filtration on $CH^p(h_X)$ defined 
in \eqref{e6.9}. Let $\{F^r\mathcal{CH}^p(C_X)\}_{r\geq 0}$ be the filtration on $\mathcal{CH}^p(C_X)$ 
given by Definition \ref{D66}. Then, the isomorphism $CH^p(X)\cong CH^p(h_X)\cong \mathcal{CH}^p(C_X)$ in 
\eqref{Z3.18}
is well behaved with respect to these filtrations, i.e., we have induced isomorphisms:
\begin{equation}
F^rCH^p(X)\cong F^rCH^p(h_X)\cong F^r\mathcal{CH}^p(C_X)\qquad \forall\textrm{ }r\geq 0
\end{equation}
\end{thm}

\begin{proof} From the proof of Proposition \ref{T3.5}, we recall the mutually inverse morphisms
$S^p:\mathcal{CH}^p(C_X)\longrightarrow CH^p(h_X)$ and $T^p:
CH^p(h_X)\longrightarrow \mathcal{CH}^p(C_X)$ defining the isomorphism
$CH^p(h_X)\cong \mathcal{CH}^p(C_X)$. 

\medskip
We choose some class $c\in F^r\mathcal{CH}^p(C_X)$ and consider the element
$S^p(c)\in CH^p(h_X)$. We now consider a morphism $f:h_Y\longrightarrow h_X$
with $Y\in SmProj/\mathbb C$ and an element $\beta\in F^sCH^k(Y)$ for some $s,k\geq 0$. 
Then, from \eqref{namo}, we have:
\begin{equation}
(S^p(c))^k_f(\beta)=c^k_{\Delta\circ f}(\beta)\in CH^{k+p}(Y)
\end{equation} where $\Delta:h_X\longrightarrow C_X$ is the morphism of presheaves
induced by the diagonal correspondence $\Delta$ from $X$ to $X$. Since 
$c\in F^r\mathcal{CH}^p(C_X)$ and $\beta \in F^sCH^k(Y)$, it follows from 
\eqref{Z4.21} that $(S^p(c))^k_f(\beta)=c^k_{\Delta\circ f}(\beta)\in F^{r+s}CH^{k+p}(Y)$. 
Hence, $S^p(c)\in F^rCH^p(h_X)$. 

\medskip
Conversely, choose $\alpha\in F^rCH^p(h_X)\cong F^rCH^p(X)$. We want to show that $T^p(\alpha)\in F^r
\mathcal{CH}^p(C_X)$. For this, we consider some $Z\in SmProj/\mathbb C$
along with a correspondence $z\in Cor(Z,X)$ defining a morphism $z:h_Z\longrightarrow C_X$. 
From \eqref{lka}, it is clear that it suffices to show that the pullback $z^*(\alpha)$ of $\alpha$ along the correspondence 
$z$ gives us an element of $F^rCH^p(Z)$. We consider $Z\times X$ and let
$p_Z:Z\times X\longrightarrow Z$ and $p_X:Z\times X\longrightarrow X$ be the respective
coordinate projections. Then, as in \eqref{3e1}, we have:
\begin{equation}\label{Z4.24}
z^*(\alpha)=p_{Z*}(z\cdot p_X^*(\alpha))
\end{equation} Since $X$ is also proper, it follows from Chow's lemma that there exists 
a birational morphism $q:X'\longrightarrow X$ with $X'$ projective. Since $X'\longrightarrow Spec(\mathbb C)$
and $X\longrightarrow Spec(\mathbb C)$ are both proper, it follows that $q:X'\longrightarrow X$
is proper. 

\medskip
We now apply resolution of singularities to $X'$. Then, (see, for instance, \cite[Theorem 3.27]{Kollar}), we can find a birational projective morphism $q':\tilde{X}
\longrightarrow X'$ with $\tilde{X}$ smooth. Then, since  $X'$ and ${q'}:\tilde{X}
\longrightarrow X'$ are projective, it follows that $\tilde{X}$ is projective. Hence, 
$\tilde{X}\in SmProj/\mathbb C$. Further, since  $q$ is proper and birational, it follows
that $q\circ {q'}:\tilde{X}\longrightarrow X$ is proper and birational. We set $\tilde{q}:=q\circ q':\tilde{X}
\longrightarrow X$ and
$\tilde{q}_Z:Z\times \tilde{X}\longrightarrow Z\times X$. We now rewrite
\eqref{Z4.24} as:
\begin{equation}\label{Z4.25}
z^*(\alpha)=p_{Z*}(z\cdot p_X^*(\alpha))=p_{Z*}(z\cdot (\tilde{q}_{Z*}\tilde{q}_Z^*p^*_X(\alpha)))=p_{Z*}\tilde{q}_{Z*}(
\tilde{q}_Z^*(z)\cdot \tilde{q}_Z^*(p_X^*(\alpha)))
\end{equation} Since $\alpha\in F^rCH^p(h_X)$, it follows from Proposition \ref{P64}
that $\tilde{q}_Z^*(p_X^*(\alpha)) \in F^rCH^p(h_{Z\times \tilde{X}})$. We note that   
$Z\times \tilde{X}\in SmProj/\mathbb C$ and hence the filtration $\{F^rCH^p(h_{Z\times \tilde{X}})\}_{r\geq 0}$ on
$CH^p(h_{Z\times \tilde{X}})$ coincides with 
Saito's filtration on $CH^p(Z\times \tilde{X})$. Since both $Z\times \tilde{X}$ and $Z$ lie in $SmProj/\mathbb C$ and  Saito's filtration as defined
on the Chow groups of objects in $SmProj/\mathbb C$ is well behaved with respect to products, pullbacks and pushforwards,
it follows from \eqref{Z4.25} that $z^*(\alpha)\in F^rCH^p(Z)$. Hence, $T^p(\alpha)\in F^r\mathcal{CH}^p(C_X)$.

\medskip
Thus, we have shown that the morphisms $S^p:\mathcal{CH}^p(C_X)\longrightarrow CH^p(h_X)$ and $T^p:
CH^p(h_X)\longrightarrow \mathcal{CH}^p(C_X)$ are both compatible with the respective filtrations
on $CH^p(h_X)$ and $\mathcal{CH}^p(C_X)$. Since $S^p$ and $T^p$ are inverses of each other,
this proves the result. 
\end{proof}

\medskip
We now want to prove a result analogous to Proposition \ref{P68} for ind-schemes. For this, we recall the notion
of an ind-scheme that is ind-proper.

\medskip
\begin{defn}\label{D70} (see \cite[$\S$ 1.3.4]{Vag}) Let $\mathcal X\in Ind-(Sm/\mathbb C)$ be an ind-scheme. Then, we say that
$\mathcal X$ is ind-proper if the presheaf $\mathcal X$ can be represented as a filtered
colimit of a system of representable presheaves on $Sm/\mathbb C$:
\begin{equation}
\mathcal X=\underset{i\in I}{colim}\textrm{ }h_{X_i}
\end{equation} where each scheme  $X_i\in Sm/\mathbb C$ is also proper.
\end{defn}

\medskip

\begin{lem}\label{L72} (a) Let $\{\mathcal F_i\}_{i\in I}$ be a filtered inductive system of objects in $Presh(Sm/\mathbb C)$. Let $\mathcal F:=colim_{i\in I}
\mathcal F_i$. Then, for each $r\geq 0$, we have isomorphisms:
\begin{equation}
F^rCH^p(\mathcal F)\cong lim_{i\in I}F^rCH^p(\mathcal F_i) \qquad \forall \textrm{ }p\in 
\mathbb Z
\end{equation}

\medskip
(b) Let $\{\mathcal G_i\}_{i\in I}$ be a filtered inductive 
system of objects in $Sh_{et}(Cor_{\mathbb C})$. Let $\mathcal G:=colim_{i\in I}
\mathcal G_i$. Then, for each $r\geq 0$, we have isomorphisms:
\begin{equation}
F^r\mathcal{CH}^p(\mathcal G)\cong lim_{i\in I}F^r\mathcal{CH}^p(\mathcal G_i) \qquad \forall \textrm{ }p\in 
\mathbb Z
\end{equation}
\end{lem}

\begin{proof} (a) Choose $p\geq 0$. We consider the canonical morphisms $f_i:\mathcal F_i\longrightarrow 
\mathcal F$, $i\in I$ to the colimit $\mathcal F$. From Lemma \ref{sinha}, we already know that
we have an isomorphism:
\begin{equation}
CH^p(\mathcal F)\cong lim_{i\in I}CH^p(\mathcal F_i) 
\end{equation} For any given $r\geq 0$, if we choose some $c\in F^rCH^p(\mathcal F)$, 
it follows from Proposition \ref{P64} that each $c_i:=f_i^*(c)\in CH^p(\mathcal F_i)$ actually
lies in $F^rCH^p(\mathcal F_i)$. 

\medskip Conversely, let $c\in CH^p(\mathcal F)$ be a class such
that $c_i=f_i^*(c)\in F^rCH^p(\mathcal F_i)$ for each $i\in I$. We consider a morphism of presheaves $y:h_Y
\longrightarrow \mathcal F$ (for some $Y\in SmProj/\mathbb C$) determining an element $y\in
\mathcal F(Y)=colim_{i\in I}\mathcal F_i(Y)$. Then, there exists some $i_0\in I$ and 
$y_{i_0}\in \mathcal F_{i_0}(Y)$ such that
$f_{i_0}(Y)(y_{i_0})=y$ where $f_{i_0}(Y):\mathcal F_{i_0}(Y)\longrightarrow 
\mathcal F(Y)$ is the morphism induced by $f_{i_0}:
\mathcal F_{i_0}\longrightarrow \mathcal F$. Then, we know from the proof of Lemma \ref{sinha} that
\begin{equation}\label{Z4.30}
c_y^k=(c_{i_0})^k_{y_{i_0}}:CH^k(Y)\longrightarrow CH^{k+p}(Y)\qquad \forall\textrm{ }k\in \mathbb Z
\end{equation} Since $c_{i_o}\in F^rCH^p(\mathcal F_{i_0})$, it follows from 
\eqref{Z4.30} and Definition \ref{D62} that
\begin{equation}
Im(c^k_y|F^sCH^k(Y))=Im((c_{i_0})^k_{y_{i_0}}|F^sCH^k(Y))\subseteq F^{s+r}CH^{k+p}(Y)\qquad \forall
\textrm{ }s\geq 0
\end{equation} Hence, $c\in F^rCH^p(\mathcal F)$. 

\medskip
(b) This is proved in a manner similar to part (a). 
\end{proof}

\medskip
Let $\mathcal X\in Ind-(Sm/\mathbb C)$ be an ind-scheme. Then, as mentioned before, it follows from 
Definition \ref{D62} that we have an extension of Saito's filtration to the 
Chow groups of the ind-scheme $\mathcal X$ (for each $p\geq 0$):
\begin{equation}\label{Z4.32}
CH^p(\mathcal X)=F^0CH^p(\mathcal X)\supseteq F^1CH^p(\mathcal X)\supseteq F^2CH^p(\mathcal X)
\supseteq \dots 
\end{equation} If we consider the presheaf with transfers $Cor(\mathcal X)$ associated to 
$\mathcal X$, it follows from Definition \ref{D66} that we have an extension of Saito's filtration
to the Chow groups of $Cor(\mathcal X)$ (for each $p\geq 0$):
\begin{equation}\label{Z4.33}
\mathcal{CH}^p(Cor(\mathcal X))=F^0\mathcal{CH}^p(Cor(\mathcal X))\supseteq F^1\mathcal{CH}^p(Cor(\mathcal X))\supseteq F^2\mathcal{CH}^p(Cor(\mathcal X))
\supseteq \dots 
\end{equation} From Proposition \ref{HHH}, we already know that $CH^p(\mathcal X)\cong \mathcal{CH}^p(Cor(\mathcal X))$. We will now show that when $\mathcal X$ is ind-proper, we have $F^rCH^p(\mathcal X)\cong F^r\mathcal{CH}^p(Cor(\mathcal X))$, $\forall$ $ r\geq 0$. This result is the analogue of Proposition \ref{P68} for ind-schemes. 

\medskip
\begin{thm}\label{P74} Let $\mathcal X\in Ind-(Sm/\mathbb C)$ be an ind-scheme that is also ind-proper. Choose
some $p\geq 0$. 
Let $\{F^rCH^p(\mathcal X)\}_{r\geq 0}$ be the descending filtration on the Chow group $CH^p(\mathcal X)$ as in \eqref{Z4.32}.
Let $\{F^r\mathcal{CH}^p(Cor(\mathcal X))\}_{r\geq 0}$ be the filtration on the Chow group 
of $Cor(\mathcal X)$ as in \eqref{Z4.33}. Then, the isomorphism $CH^p(\mathcal X)\cong \mathcal{CH}^p(Cor(\mathcal X))$ 
in \eqref{Devendra}
is well behaved with respect to these filtrations, i.e., we have isomorphisms:
\begin{equation}
F^rCH^p(\mathcal X)\cong F^r\mathcal{CH}^p(Cor(\mathcal X))\qquad \forall\textrm{ }r\geq 0
\end{equation}

\end{thm}

\begin{proof} We know that $\mathcal X$ is ind-proper. Hence, we can express $\mathcal X$
as a filtered colimit
\begin{equation}\label{Z4.35}
\mathcal X=\underset{i\in I}{colim}\textrm{ }h_{X_i}
\end{equation} where each scheme $X_i\in Sm/\mathbb C$ is also proper. From Lemma \ref{L72}(a), it now
follows that for any $r\geq 0$, we have:
\begin{equation}\label{pankaja}
F^rCH^p(\mathcal X)=\underset{i\in I}{lim}\textrm{ }F^rCH^p(h_{X_i})
\end{equation} We consider the presheaf with transfers $Cor(\mathcal X)$ associated
to $\mathcal X$. From \eqref{gopinath}, we know that the functor $Cor$ is a left adjoint and hence it preserves
colimits. From \eqref{Z4.35}, it now follows that $Cor(\mathcal X)$ may be expressed as the colimit:
\begin{equation}\label{pramod}
Cor(\mathcal X)=\underset{i\in I}{colim}\textrm{ }Cor(h_{X_i})=\underset{i\in I}{colim}\textrm{ }C_{X_i}
\end{equation} For each object $X_i\in Sm/\mathbb C$, we have already noted before that $C_{X_i}$ is
an \'{e}tale sheaf with transfers, i.e., $C_{X_i}\in Sh_{et}(Cor_{\mathbb C})$. Applying Lemma \ref{L72}(b), it follows
that for any $r\geq 0$, we have:
\begin{equation}\label{ArjunM}
F^r\mathcal{CH}^p(Cor(\mathcal X))=\underset{i\in I}{lim}\textrm{ }F^r\mathcal{CH}^p(C_{X_i})
\end{equation} Finally, since each $X_i$ is proper, it follows from Proposition \ref{P68} that
$F^rCH^p(h_{X_i})\cong F^r\mathcal{CH}^p(C_{ X_i})$ for every $r\geq 0$, $i\in I$. Then, comparing
\eqref{pankaja} and \eqref{ArjunM}, we have the result.

\end{proof}

\medskip

\end{document}